\documentclass{article}
\usepackage{graphicx}
\usepackage{amsfonts}
\usepackage{amsmath}
\usepackage[utf8]{inputenc}
\usepackage{xy} \xyoption{all}
\usepackage{tikz}
\usetikzlibrary{shapes,arrows}
\usepackage{ wasysym}
\usepackage{amssymb}
\usepackage{mathrsfs}

\usepackage{multicol}

\newcommand{\hookdownarrow}{\mathrel{\rotatebox[origin=c]{-90}{$\hookrightarrow$}}}

\newtheorem{theorem}{Theorem}

\newtheorem{corollary}[theorem]{Corollary}

\newtheorem{example}[theorem]{Example}
\newtheorem{examples}[theorem]{Examples}

\newtheorem{lemma}[theorem]{Lemma}

\newtheorem{proposition}[theorem]{Proposition}
\newtheorem{remark}[theorem]{Remark}

\newtheorem{fact}[theorem]{Fact}
\newenvironment{proof}[1][Proof]{\noindent\textbf{#1.} }{\ \rule{0.5em}{0.5em}}
\begin{document}
\title{Classification of Leavitt Path Algebras with Gelfand-Kirillov Dimension $< $ 4 up to Morita Equivalence}
\author{Ayten Ko\c{c} $^{*}\> $, $\>$  Murad  \"{O}zayd\i n $^{**}$\\ \\}

%\documentclass[border=0.2cm]{standalone}
 
% Required packages
%\usepackage{tikz}
\usetikzlibrary{shapes,positioning}

 \maketitle
 
\medskip

\begin{abstract}
  Leavitt path algebras are associated to di(rected )graphs and there is a combinatorial procedure (the reduction algorithm) making the digraph smaller while preserving the Morita type. We can recover the vertices and most of the arrows of the completely reduced digraph from the module category of a Leavitt path algebra of polynomial growth.  We give an explicit classification of all irreducible representations of when the coefficients are a commutative ring with 1. We define  a Morita invariant filtration of the module category by Serre subcategories and as a consequence we obtain a Morita invariant (the weighted Hasse diagram of the digraph) which captures the poset of the sinks and the cycles of $\Gamma$, the Gelfand-Kirillov dimension and more.  When the Gelfand-Kirillov dimension of the Leavitt path algebra is less than 4, the weighted Hasse diagram (equivalently, the complete reduction of the digraph) is a complete Morita invariant. \\
\end{abstract}

\medskip

\medskip
\section{Introduction}
\medskip

Leavitt Path algebras (LPAs) have been a very active area of research recently, however their module theory is still in its infancy  \cite[Introduction]{r20}. This paper  is a contribution to this area from a categorical perspective.\\

The Leavitt path algebra $L(\Gamma)$ of a di(rected )graph $\Gamma$ was defined (many decades after Leavitt's seminal work  \cite{lea62}, via a detour through functional analysis) by Abrams, Aranda Pino \cite{aa05} and by Ara, Moreno, Pardo \cite{amp07} (independently and essentially simultaneously) as an algebraic analog of a graph $C^*$-algebra. In addition to the algebras $L(1,n)$ of Leavitt \cite{lea62}  these include (sums of) matrix algebras (over fields or Laurent polynomial algebras), algebraic quantum discs and spheres, and many others. The important subclass of Leavitt path algebras of polynomial growth were identified as coming from finite digraphs whose cycles are pairwise disjoint and then studied by Alahmedi, Alsulami, Jain, Zelmanov \cite{aajz12}, \cite{aajz13}.  \\

In \cite{ko1} we showed that the category of (unital) $L_{\mathbb{F}}(\Gamma)$-modules  is equivalent to a full subcategory of quiver representations satisfying a natural isomorphism condition when $\Gamma$ is a  row-finite digraph.  Here we remark that the same result is valid when the coefficients are a commutative ring ${\bf k}$ with 1. (We always denote a field by $\mathbb{F}$ and a commutative ring with 1 by ${\bf k}$.) Similarly, the explicit Morita equivalence given by an effective combinatorial (reduction) algorithm on a  digraph  $\Gamma$ originally given in \cite{ko2} to classify all finite dimensional modules of the Leavitt path algebra $L(\Gamma)$  also generalizes  when the coefficient field $\mathbb{F}$ is replaced with ${\bf k}$. In addition to these two results we give the relevant definitions and some fundamental facts in the next section on the preliminaries, in the generality of the coefficients being a commutative ring ${\bf k}$.\\

In Section 3 we focus on a finite digraph $\Gamma$ whose cycles are pairwise disjoint, a necessary and sufficient condition for $L_{\mathbb{F}}(\Gamma)$ to have polynomial growth. First we define a convenient ${\bf k}$-basis for $L_{\bf k}(\Gamma)$ in Theorem \ref{basis} which is useful for giving a combinatorial formula of the Gelfand-Kirillov dimension (Theorem \ref{height}, Theorem \ref{k-height}) as well as characterizing some important Serre subcategories of the module category $Mod_{L_{\bf k}(\Gamma)} $ .\\

We  give a fairly explicit classification of simple $L_{\bf k}(\Gamma)$-modules in Theorem \ref{Classification 1} of Section 4. When the coefficients are a field,  simple modules were identified as Chen modules and their generalizations by Ara and Rangaswamy \cite{ar14}. Our explicit form of the simple modules is utilized  to compute their extensions, which is needed in the next section.\\

 In Section 5 we work towards classifying $L_{\mathbb{F}}(\Gamma)$ when $\Gamma$ is a finite digraph whose cycles are pairwise disjoint, up to Morita equivalence. To this end we define filtrations of $\Gamma$ by subgraphs, $L_{\mathbb{F}}(\Gamma)$ by invariant graded ideals and $Mod_{L_{\mathbb{F}}(\Gamma)}$ by Serre subcategories. We obtain a Morita invariant polynomial (Theorem \ref{polinom}) whose coefficients are defined combinatorially and whose degree is the Gelfand Kirillov dimension of  $L_{\mathbb{F}}(\Gamma)$. We show that the poset of the sinks and the cycles of $\Gamma$ is also a Morita invariant. We define the \textit{weighted Hasse diagram} of $\Gamma$, which is a Morita invariant containing all the Morita invariants 
 just mentioned, and more. When the Gelfand-Kirillov dimension of $L_{\mathbb{F}}(\Gamma)$ is less then 4, the weighted Hasse diagram (equivalently, the complete reduction of $\Gamma$) turns out to be a complete Morita invariant (Theorem \ref{<4}).   \\

\medskip
\section{Preliminaries}
\subsection{Digraphs}
\medskip 
A {\em di}({\em rected} ){\em graph} $\Gamma$ is a four-tuple $(V,E,s,t)$ where $V$ is the set of vertices, $E$ is the set of arrows (directed edges), $s$ and $t:E \longrightarrow V$ are the source and  the target functions. The digraph $\Gamma$ is \textit{finite} if $E$ and $V$ are both finite. $\Gamma$ is \textit{row-finite} if $s^{-1}(v)$ is finite for all $v$ in $V$. If $s^{-1}(v)= \emptyset$ then  the vertex $v$ is a \textit{sink}; if $t^{-1}(v)= \emptyset$ then it is a \textit{source}. If $t(e)=s(e)$ then $e$ is  a \textit{loop}. If $W\subseteq V$ then $\Gamma_{W}$ denotes the \textit{full subgraph} on $W$, that is, $\Gamma_W=(W, s^{-1}(W) \cap t^{-1}(W), s\vert_{_W}, t\vert_{_W})$.\\

A \textit{path} $p$ of length $n>0$ is a sequence of arrows $e_{1}\ldots e_{n}$ such
that $t(e_{i})=s(e_{i+1})$ for $i=1,\ldots ,n-1$. The source of $p$ is $s(p):=s(e_{1})$  and the target of $p$ is $t(p ):=t(e_{n})$.  A path of length 0 consists of a single vertex $v$ where $s(v):=v $ and $t(v) := v$. We will denote the length of $p$ by $l(p)$. 
The set $int(p)$ of \textit{internal vertices} of a path $p=e_1e_2 \cdots e_n$ is $\{se_1,se_2,\cdots , se_n \}$ and $int(p)=\emptyset$ if $l(p)=0$, that is, $p$ is a vertex. 
An \textit{infinite path} $p$ is an infinite sequence of arrows $e_1e_2 \cdots e_k \cdots $
 with $t(e_i)=s(e_{i+1})$ for each $i$; now $s(p)=s(e_1)$ but $t(p)$ is not defined. An arrow $e$ (respectively, a vertex $v$) is said to be on $p$ if $e$ is one of the arrows in the sequence defining $p$ (respectively, if $v=s(p)$ or $v=t(e)$ for any arrow $e$ on $p$).  An \textit{exit} of a path $p$ is an arrow $e$ which is not on $p$ but $s(e)$ is on $p$.  A path $C=e_1e_2 \cdots e_n$ with $n>0$ is a \textit{cycle} if $s(C )=t(C )$ and $s(e_{i})\neq s(e_{j})$ for $i\neq j$. We consider the cycles $e_1e_2 \cdots e_n$ and $e_2e_3 \cdots e_ne_1$ equivalent.
 The digraph $\Gamma$ is \textit{acyclic} if it has no cycles. 
  $Path(\Gamma)$ denotes the set of paths in $\Gamma$, $P_w:=\{p \in Path(\Gamma) \> \vert \> tp=w \notin int(p) \}$ and $P_w^v :=\{ p\in P_w \> \vert \> sp=v \}$.\\

There is a pre-order $\leadsto$ on the set of vertices $V$ of $\Gamma$: $u \leadsto v $ if and only if   there is a path from $ u$ to   $ v$. This pre-order defines an equivalence relation on $V$: $u\sim v$ if and only if $u\leadsto v$ and $v\leadsto u$. There is an induced partial order on equivalence classes, which we will also denote by $[u] \leadsto [v]$. \\

When the cycles of $\Gamma$ are pairwise disjoint, the preorder $\leadsto$ defines a partial order on the set of sinks and cycles in $\Gamma$. For each cycle $C$ let's fix a vertex $v_C$ on $C$ and let $U:=\{ w\in V \> \vert \> w \textit{ is a sink in } \Gamma \> \} \cup \{ v_C \> \vert \> C \textit{ is a cycle in } \Gamma \}$. Clearly there is an order preserving one-to-one correspondence between $U$ and the set of sinks and cycles in $\Gamma$.\\

The \textbf{predecessors} of $w$ in $V$ is $V_{\leadsto w} := \{ v \in V \mid v \leadsto w\} $. (This set is also denoted by $M(w)$ in the literature.) If $v$ and $w$ are two vertices on a cycle $C$ then $V_{\leadsto v}=V_{\leadsto w}$,  so $V_{\leadsto C} :=V_{\leadsto v}$ is well-defined. Let $\Gamma_{\leadsto w}$ and $\Gamma_{\leadsto C}$ be the full subgraphs on $V_{\leadsto w}$ and $V_{\leadsto C}$, respectively. \\

 \medskip
\subsection{Leavitt Path Algebras}
\medskip

Given a digraph $\Gamma,$ the \textit{extended digraph} of $\Gamma$ is $\tilde{\Gamma} := (V,E \sqcup E^*, s~,t~)$ where $E^* :=\{e^*~|~ e\in E \}$,  the functions $s$ and $t$ are  extended as $s(e^{\ast}):=t(e)$ and $t(e^{\ast }):=s(e) $ for all $e \in E$. Thus the dual arrow $e^*$ has the opposite orientation of $e$. We extend $*$ to an operator defined on all paths of $\tilde{\Gamma}$: Let $v^*:=v$ for all $v$ in $V$, $(e^*)^*:=e $ for all $e$ in $E$ and $p^*:= e_n^* \ldots e_1^*$ for a path $p=e_1 \ldots e_n$ with $e_1, \ldots , e_n $ in $E \sqcup E^* $. In particular $*$ is an involution, that is, $**=id$. \\

The \textbf{ Leavitt path algebra} of a digraph $\Gamma$ with coefficients in a commutative ring ${\bf k}$ with 1, as defined in \cite{aa05}, \cite{amp07} and \cite{t11} is the ${\bf k}$-algebra $L_{\bf k}(\Gamma)$ generated by $V \sqcup E \sqcup E^*$ satisfying:\\
\indent(V)  $\quad \quad vw ~=~ \delta_{v,w}v$ for all $v, w \in V ,$ \\
\indent($E$)  $\quad \quad s ( e ) e  = e=e t( e)  $ for all $e \in E\sqcup E^*$, \\
\indent(CK1) $\quad e^*f ~ = ~ \delta_{e,f} ~t(e)$ for all $e,f\in E$, \\
\indent(CK2) $\quad v~  = ~ \sum_{s(e)=v} ee^*$  for all $v$ with $0< \vert s^{-1}(v) \vert < \infty $\\
\noindent
where $\delta_{i,j}$ is the Kronecker delta.
$L_{\bf k}(\Gamma)$ is a $*$-algebra since these relations are compatible with the involution $*$ which defines an anti-automorphism on $L_{\bf k}(\Gamma)$: for all $a$, $b$ in $L_{\bf k}(\Gamma)$ we have $(ab)^*=b^*a^*$. \\

We sometimes suppress the subscript ${\bf k}$ in the notation $L_{\bf k}(\Gamma)$. If the digraph $\Gamma$ is fixed and clear from the context we may abbreviate $L(\Gamma)$ to $L$. From now on we will omit the parentheses when the source and target functions $s$, $t$ are applied, to reduce notational clutter. \\

The relations (V) simply state that the vertices are mutually orthogonal idempotents. The relations (E) implies that $e \in seLte$ and $e^* \in teLse$ for every $e \in E$. The Leavitt path algebra $L$ as a vector space is $\bigoplus vLw$ where the sum is over all pairs $(v,w) \in V \times V$ since $V\sqcup E\sqcup E^*$ generates $L$. If we only impose the relations (V) and ($E$) then we obtain ${\bf k}\tilde{\Gamma}$, the \textbf{path }(or \textbf{quiver}) \textbf{algebra} of the extended digraph $\tilde{\Gamma}$ : The paths in $\tilde{\Gamma}$ form a basis of the free ${\bf k}$-module ${\bf k}\tilde{\Gamma}$. The multiplicative structure of ${\bf k}\tilde{\Gamma}$ is given by:  the product $pq$ of two paths $p$ and $q$ is their concatenation if $tp=sq$ and 0 otherwise, extended linearly. 
$L_{\bf k}(\Gamma)$ is a quotient of $ {\bf k}  \tilde{\Gamma}$ by the ideal generated by the Cuntz-Krieger relations (CK1), (CK2). The algebras $ {\bf k}\tilde{\Gamma}$ and $L_{\bf k}(\Gamma)$ are unital if $V$ is finite, in which case the sum of all the vertices is 1. \\

\begin{fact} \label{fact}

Let $p$ and $q$ be paths and $C$ be  a cycle with no exit in $\Gamma$.\\ 
\noindent
(i) $p^*q=0$ unless 
   $q$ is an initial segment of $p$  (i.e., $p=qr$) or $p$ is an initial segment of $q$. \\
\noindent
(ii) $L_{\bf k}(\Gamma)$ is spanned by $\left\{pq^* \vert \> p,\> q \in Path \Gamma, \> tp=tq \right\}$ as a ${\bf k}$-module.\\
\noindent
(iii) If $C$ is a cycle with no exit then $CC^*=sC=C^*C$. 
\end{fact}

\begin{proof}
(i) If $p$ is not an initial segment of $q$ and $q$ is not an initial segment of $p$ then using the relations (CK1), (E) and (V) when necessary, we can simplify $p^*q$ until we get a term $e^*f$ with $e\neq f$, which is 0 by (CK1).\\

(ii) If $p=qr$ then $p^*q$ simplifies to $r^*$, similarly if $q=pr$ then $p^*q=r$. Hence (i) implies that multiplying terms of the form $pq^*$ we get 0 or another term of this form. Also, if $tp\neq tq$ then $pq^*=0$ by (V) and (E). Thus $L_{\bf k}(\Gamma)$ is generated as a ${\bf k}$-module by $pq^*$ with $p, \> q$ in $Path\Gamma$ and $tp=tq$.\\

(iii) As in (ii), $C^*C=tC=sC$ after applying (CK1), (V) and (E) repeatedly. Since $C$ has no exit,  $se=ee^*$ for each arrow $e$ on $C$ by (CK2). Thus $CC^*=sC$ using (CK2) and (E).
\end{proof}\\

 $L(\Gamma)$ is a $\mathbb{Z}$-graded $*$-algebra: the $\mathbb{Z}$-grading on $ \mathbb{F}  \tilde{\Gamma}$ is given by $deg(v)=0$ for $v$ in $V$, $deg(e)=1$ and $deg(e^*)=-1$ for $e$ in $E$. This defines a grading on $L(\Gamma)$ since all the relations are homogeneous. The linear extension of $*$ on paths  to $L(\Gamma)$ induces a grade-reversing involutive anti-automorphism (that is, $deg(a^*)=-deg(a)$ for all homogeneous $a$ in $L(\Gamma)$). Hence the categories of left modules and right modules are equivalent for $L(\Gamma)$.\\

 We may consider $G$-gradings on $L(\Gamma)$ for any group $G$, with the generators $V \sqcup E \sqcup E^*$ being homogeneous. Since $v^2=v$ and $e^*e=te$ we have: (i) $ \lvert  v \rvert_{_{G}} =1_{_{G}}$ and (ii) $\lvert e^*\rvert_{_G}= \lvert e\rvert^{-1}_{_G}$. Conversely, any function from $V \sqcup E \sqcup E^*$ to $G$ satisfying (i) and (ii) defines a $G$-grading on $L(\Gamma)$ as the remaining relations are homogeneous. A morphism (or a refinement) from a $G$-grading to an $H$-grading on the algebra $A$ is given by a group homomorphism $\phi:G \longrightarrow H$ such that for all $h \in H$, $A_h= \oplus_{\phi (g)=h} A_g$ where $A_g:= \{a \in A  :\lvert a\rvert_{_{G}}=g \} \cup \{0\} $. There is a universal (or initial) $G$-grading on $L(\Gamma)$ given by $G=F_E$, the free group on $E$, which is a refinement of all others:

\begin{proposition}
Let $G:= F_{E}$ be the free group on the set of arrows. The $G$-grading defined by $\vert v\vert_{_G}=1$, $\vert e\vert_{_G}= e$ and $\vert e^*\vert_{_G}=e^{-1}$ is an initial object in the category of $G$-gradings of 
$L(\Gamma)$ with the generators $V \sqcup E \sqcup E^*$ being homogeneous.
\end{proposition}

\begin{proof}
For any $H$-grading let $\phi: G \rightarrow H$ be the homomorphism given by $\phi(e)= \vert e\vert_{_H}$.
\end{proof}\\

The universal grading will be used later in the construction of a ${\bf k}$-basis for $L_{\bf k}(\Gamma)$. Another application is Proposition \ref{Path}(i) showing that the path algebra ${\bf k}\Gamma$ may be identified with a subalgebra of $L_{\bf k}(\Gamma)$. When ${\bf k}$ is a field, a (different) proof of this basic fact was originally given in \cite[Lemma 1.6]{goo09}.\\

A subset $H$ of $V$ is \textit{hereditary} if $v$ is a \textit{successor} of $u$ (that is, $u\leadsto v$) and $u\in H$ implies that $v\in H$. The subset $H$ is \textit{saturated} if  $te \in H $ for all $e$ in $s^{-1}(v)$ implies that $v\in H$ for each non-sink $v$ with $s^{-1}(v)$ finite. The \textit{hereditary saturated closure} of a subset $X$ of $V$, denoted by $\Bar{X}$, is the smallest hereditary saturated subset of $V$ containing $X$. The hereditary closure of $X\subseteq V$ is $\{ v\in V \>\vert \> x \leadsto v \textit{ for some } x \in X\}$. When $\Gamma$ is row-finite, the hereditary saturated closure of $X\subseteq V$ consists of all vertices $v\in V$ such that every path starting at $v$ and ending at a sink and every infinite path starting at $v$ meets a successor of some $x \in X$.\\

The full subgraph on a hereditary saturated $H \subseteq V$ is denoted by $\Gamma_H$. We obtain the subgraph $\Gamma_{/H}$ by deleting all the vertices in $H$ and all arrows touching them, so $\Gamma_{/H}$ is the full subgraph on $V\setminus H$.\\

When ${\bf k}$ is a field, the following result is well-known   \cite[Lemma 2.4.3 and Corollary  2.4.13(i) ]{aam17}.

\begin{proposition} \label{hereditary} 
Let $\Gamma$ be a row-finite digraph and ${\bf k}$ a commutative ring with 1. \\
(i) If $I$ is an ideal of $L_{\bf k}(\Gamma)$  and $\lambda \in {\bf k}$ then $\{v\in V \> \vert \> \lambda v \in I \}$ is a hereditary saturated subset of $V$.  \\      
(ii) If $H$ is a hereditary saturated subset of $V$ and $\Gamma_{/H}$ as above then $L_{\bf k}(\Gamma_{/H}) \cong L_{\bf k}(\Gamma)/ I$ as $\mathbb{Z}$-graded $*$-algebras 
where $I:=(H)$ is the ideal generated by $H$. Also, $vL_{\bf k}(\Gamma_{/H}) \cong vL_{\bf k}(\Gamma)/ vI$ for all $v \notin H$.
\end{proposition}
\begin{proof}
(i) If $e \in E$ with $se =v$
and $\lambda v \in I$ then $
\lambda te =\lambda        e^*e=e^*(\lambda v)e \in I$, hence $\{v\in V \> \vert \> \lambda v \in I \}$ is hereditary. If $v$ is not a sink and $\lambda te \in I$ for all $e$ with $se =v$ then $\lambda v=\lambda \sum ee^*=\sum e(\lambda te)e^* \in I$, hence $\{v\in V \> \vert \> \lambda v \in I \}$ is saturated. \\

(ii) We define a $*$-homomorphism from $L_{\bf k}(\Gamma_{/H})$ to $ L_{\bf k}(\Gamma)/ (H)$ by sending all vertices $v$ to $v+(H)$ and all arrows $e$ to $e+(H)$. The relations $(V)$, $(E)$ and $(CK1)$ are automatically satisfied. If $v\notin H$ is not a sink in $\Gamma$ then $v$ is not a sink $\Gamma_{/H}$ since $H$ is saturated. In $L_{\bf k}(\Gamma)$   $$v=\sum_{se=v \>, \> 
te \notin H} ee^* +\sum_{sf=v\>, \>
tf \in H} ff^*$$ 

by $(CK2)$ in $L_{\bf k}(\Gamma)$. Since the second sum is in $(H)$ we see that $(CK2)$ in $L_{\bf k}(\Gamma_{/H})$ is satisfied and thus we have a homomorphism. \\

We define a $*$- homomorphism from $L_{\bf k}(\Gamma)$  to $L_{\bf k}(\Gamma_{/H})$ by sending $v$ to $v$ if $v \notin H$ and to 0 otherwise, similarly $e$ to $e$ if $e$ is an arrow in $\Gamma_{/H}$ and to 0 otherwise. The relations $(V)$ and $(E)$ are immediate to verify. Most cases of $(CK1)$ are also easy to check. For the case of $e^*e=te $, if $te \notin H$ then $se \notin H$ since $H$ is hereditary, hence $e$ is in $\Gamma_{/H}$ and the relation holds. When $v$ is not a sink in $\Gamma$ then $v$ is not a sink in $\Gamma_{/H}$, breaking up the right-hand side of $(CK2)$ as above shows that the relation is satisfied. The ideal $(H)$ is in the kernel of this homomorphism, so we have the induced homomorphism from $L(\Gamma)/(H)$ to $L(\Gamma_{/H})$.\\

The $*$-homomorphisms above are both $\mathbb{Z}$-graded and are inverses of each other. The last assertion follows from restricting the first homomorphism to $vL_{\bf k}(\Gamma_{/H})$.
\end{proof}\\

\begin{fact} \label{exit}
If $p=e_1e_2 \cdots e_n$ let $X_p:=\bigsqcup_{k=1}^n \{ e_1e_2\cdots e_{k-1}f \vert  \>  e_k \neq f \in E\>, \> \> sf=se_k \} \qquad $ then 
$$p^*p-pp^*= sp-pp^*=\sum_{q\in X_p} qq^*.$$
\end{fact}
\begin{proof}
$p^*p=sp$ by (CK1). We will induct on $n$. If $n=1$ this is (CK2). If $n>1$ then induction hypothesis on $p':=e_2\cdots e_n$ with 
$X_p=X_{e_1} \sqcup e_1.X_{p'}$ and $e_1(se_2)e_1^*=e_1e_1^*$ gives the desired result.  
\end{proof}\\

\begin{example} \label{example} Consider the following digraphs:\\

$ \bullet_v \qquad \qquad \bullet_u {\stackrel{e}{\longrightarrow}} \bullet_v \quad  \quad \qquad  
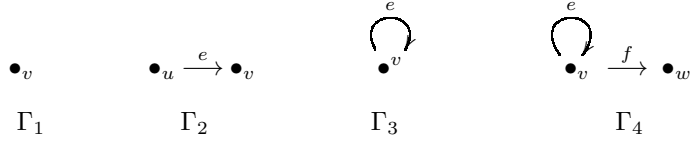
\xymatrix{{\bullet}^{v} \ar@(ul,ur) ^{e}}\quad \quad \xymatrix{  &\> \> { \bullet}_{v} \ar@(ul,ur) ^{e}} \xymatrix{ \stackrel{f}{\longrightarrow} \>   \bullet_{w} }$ 
$$\Gamma_1 \qquad \quad \quad \quad \Gamma_2 \quad \qquad \qquad  \quad \Gamma_3 \qquad \qquad  \qquad \qquad \Gamma_4 \qquad\qquad \qquad \quad $$ 

 (i) $\mathbb{F}\Gamma_1=\mathbb{F}v =L(\Gamma_1)$.\\

(ii)  $L(\Gamma_2) \cong M_2(\mathbb{F})$ where $u \leftrightarrow E_{11}\> , \quad v \leftrightarrow E_{22} \> , \quad e \leftrightarrow E_{12} \> , \quad e^* \leftrightarrow E_{21}$.\\
More generally, if  $\Gamma$ has no cycles then $L(\Gamma)$ is isomorphic to a direct sum of matrix algebras $M_n(\mathbb{F})$, each summand corresponds to a sink $w$ and $n=\vert P_w \vert $, the number of paths ending at $w$ (Proposition 3.5 \cite{aam07}, this also follows from Theorem \ref{basis} below).\\

(iii) $\mathbb{F}\Gamma_3 \cong \mathbb{F}[x] $ and $L(\Gamma_3) \cong \mathbb{F}[x,x^{-1}]$ where $v \leftrightarrow 1 , \quad e \leftrightarrow x^{-1}, \quad e^* \leftrightarrow x$. \\
More generally, if the cycles of $\Gamma$ have no exits then $L(\Gamma)$ is isomorphic to a direct sum of matrix algebras $M_n(\mathbb{F})$ or $M_n(\mathbb{F}[x,x^{-1}])$ with each summand $M_n(\mathbb{F})$ corresponding to a sink as in (ii) and each summand $M_n(\mathbb{F}[x,x^{-1}])$ corresponding to 
a cycle $C$ with $n=\vert P_{sC} \vert $ (Theorems 1.8 and 3.8 in \cite{aam08}, this also follows from Theorem \ref{basis} below).\\

(iv) $L(\Gamma_4) \cong  \mathbb{F}\langle x, y \rangle /(1-xy)$ where $x \leftrightarrow e^*+f^* ,\quad y \leftrightarrow e+f,  \quad 1\leftrightarrow v+w$ also $v \leftrightarrow yx , \quad e \leftrightarrow y^2x , \quad e^*\leftrightarrow yx^2$.
This is a graded isomorphism of $*$-algebras where $\vert x\vert=-1$, $\vert y\vert=1$ and $x=y^*$. The only proper nontrivial graded ideal of the Jacobson algebra $\mathbb{F} \langle  x, y \rangle /(1-xy)$ is generated by $1-yx \leftrightarrow w$.

\end{example}

 \medskip
\subsection{$L(\Gamma)$-modules and Quiver Representations}
\medskip

The following result identifying the category of unital $L(\Gamma)$-modules with a full subcategory of quiver representations of $\Gamma$ enables us to construct many $L(\Gamma)$-modules and to derive some fundamental properties. A right $R$-module $M$ is \textit{unital} if $ M=MR$ where $R$ may not have 1. This definition agrees with the usual definition of a unital module when $R$ has 1. All our modules will be right modules unless specified otherwise.\\

 Recall that a quiver representation is a functor from the small category (also denoted by $\Gamma$) with objects $V$ and morphisms given by paths in $\Gamma$, to the category of unital {\bf k}-modules where {\bf k} is a commutative ring with 1 (classically {\bf k} is a field \cite{dw05}). The coefficient ring ${\bf k}$ is suppressed in the statement of the following proposition. 
 
 \begin{proposition} \label{teorem} 
If $\Gamma= (V,E,s,t)$ is a row-finite digraph then $Mod_{L(\Gamma)}$, the category of unital right $L(\Gamma)$-modules, 
% of unital right $L$-modules
 is equivalent to the full subcategory of quiver representations $\rho$ of $\Gamma$ satisfying the following condition $(I)$:\\
  $$\textit{For every nonsink } v\in V, \> \> \> \>  \bigoplus_{se=v} \rho(e): \rho ( v) \longrightarrow \bigoplus\limits_{se=v } \rho(te) \textit{ is an isomorphism.}$$ The right $L(\Gamma)$-module corresponding to the quiver representation $\rho$ is $M=\bigoplus_{v\in V} \rho(v)$ as a ${\bf k}$-module and the actions of generators of  $L(\Gamma)$ are given by the compositions
  $$\cdot v : M=\bigoplus_{u \in V} \rho(u) \stackrel{pr_{\rho(v)}}{\longrightarrow} \rho(v) \hookrightarrow \bigoplus_{u \in V} \rho(u)=M ,$$
  $$\cdot e : M=\bigoplus_{v\in V} \rho(v) \stackrel{pr_{se}}{\longrightarrow }\rho(se) \stackrel{\rho(e)}{\longrightarrow }\rho(te) \hookrightarrow \bigoplus \rho(v)=M ,$$ 
  $$\cdot e^* : M=\bigoplus_{v\in V} \rho(v) \stackrel{pr_{te}}{\longrightarrow }\rho(te) \hookrightarrow   \bigoplus_{sf=se} \rho(tf) \stackrel{(\bigoplus \rho(f))^{-1}}{ \longrightarrow} \rho(se) \hookrightarrow \bigoplus_{v \in V} \rho(v) =M .$$ 
  
\end{proposition}

The proof of Proposition \ref{teorem} is given in \cite[Theorem 3.2]{ko1} for a field $\mathbb{F}$, but works for a commutative ring ${\bf k}$ with 1. As in the statement of the Proposition \ref{teorem} from now on $Mod_R$ will denote the category of unital right modules of the ring $R$. \\

We denote by ${\bf k}X$, the free ${\bf k}$-module with basis $X$.

\begin{example} \label{nonzero p} 
Let $\rho(v)={\bf k}\mathbb{Z}$ for all $v \in V$. Pick an isomorphism $$\varphi_v: \rho(v) \longrightarrow \bigoplus_{se=v} \rho(te)$$ for each non-sink $v\in V$. Let $\rho(e)$ be the composition 
$$ \rho(se)\stackrel{\varphi_{se}}{\longrightarrow} \bigoplus_{sf=se} \rho(tf) \stackrel{pr_e}{\longrightarrow} \rho(te)$$ 
and $\rho(e^*)$ be the composition
$$ \rho(te)\stackrel{\varphi_{se}}{\hookrightarrow} \bigoplus_{sf=se} \rho(tf) \stackrel{\varphi_{se}^{-1}}{\longrightarrow} \rho(se) .$$ 

Since condition (I) is satisfied by construction, we have an $L_{{\bf k}}(\Gamma)$-module.\\
\end{example}

In the dictionary between a unital $L_{\bf k}(\Gamma)$-module $M$ and a quiver representation $\rho$ of $\Gamma$ satisfying condition (I), the ${\bf k}$-module $\rho(v)$ corresponds to $Mv=\{ m \in M \> \vert \> mv=m \}$ for each vertex $v$. So $M\cong \oplus_{v \in V} Mv$ as a ${\bf k}$-module. 
The {\bf support} of $M$ is $V_M:=\{ v\in V \> \vert \> Mv \neq 0 \}$, the support subgraph of $M$, denoted by $\Gamma_M$, is the full subgraph of $\Gamma$ on $V_M$. Since a module homomorphism $\varphi : M \longrightarrow N$ is made up of ${\bf k}$-linear maps $\varphi_v: Mv \longrightarrow Nv$ for all $v \in V$, both $V_M$ and $\Gamma_M$ are isomorphism invariants. Also, if $0 \longrightarrow A \longrightarrow B \longrightarrow C \longrightarrow 0$ is a short exact sequence of $L_{{\bf k}}(\Gamma)$-modules then $V_B=V_A \cup V_C$.\\

Immediate consequences of Proposition \ref{teorem} and Example \ref{nonzero p} are:

\begin{proposition} \label{pp^*}

(i) Let $M$ be a unital $L_{\bf k}(\Gamma)$-module. For any path $p$ in Path$(\Gamma)$ the ${\bf k}$-module homomorphism $Msp \stackrel{\cdot p}{\longrightarrow} Mtp$ defined by right multiplication with $p$ is onto and the ${\bf k}$-module homomorphism $Mtp \stackrel{\cdot p^*}{\longrightarrow } Msp$ is one-to-one. \\

(ii) For any two paths $p$ and $q$ in $Path \Gamma$ with $tp=tq$ and $\lambda \in {\bf k}$, 
 $\lambda pq^*=0$ in $L_{\bf k}(\Gamma)$ if and only if $\lambda=0$. \\
 
 (iii) If $w$ is a sink then $wL_{\bf k}(\Gamma)w ={\bf k}w \cong {\bf k}$.\\

\end{proposition}
\begin{proof}
 (i) By the condition (I) in Proposition \ref{teorem}, $Mse \stackrel{\cdot e}{\longrightarrow} Mte $ is onto for all $e \in E$, hence $Msp \stackrel{\cdot p}{\longrightarrow} Mtp $ is also onto. Similarly, $Mte \stackrel{\cdot e^*}{\longrightarrow} Mse $ is one-to-one, so $Mtp \stackrel{\cdot p^*}{\longrightarrow} Msp $ is also one-to-one.\\

(ii) Given an $L_{\bf k}(\Gamma)$-module $M$ as in the Example \ref{nonzero p} above right multiplication by $pq^*$ with $tp=tq$ defines a ${\bf k}$-module homomorphism from $Msp$ to $Msq$ whose image is ${\bf k}\mathbb{Z}$, a free ${\bf k}$-module of infinite rank by (i). Hence $\lambda pq^*=0$ if and only if $\lambda=0$.\\

(iii) Since $w$ is a sink and $wL_{\bf k}(\Gamma)w$  is spanned by $pq^*$ with $sp=w=sq$, we see that $p=q=w=pq^*$. Hence $wL_{\bf k}(\Gamma)w ={\bf k}w \cong {\bf k}$ by (ii) above.
\end{proof}\\

 When ${\bf k}$ is a field, part (i) of the following proposition is in \cite[Lemma 1.6]{goo09} and part (ii) is in  \cite[Lemma 2.2.7]{aam17}. \\
 
\begin{proposition} \label{Path}
(i) The homomorphism from ${\bf k}\Gamma$ to $L_{\bf k}(\Gamma)$ sending every vertex and every arrow to itself is one-to-one.\\
(ii) If $C$ is a cycle with no exit then $vL_{{\bf k}}(\Gamma)v \cong {\bf k}[x,x^{-1}]$ where $v=sC$ and $x \leftrightarrow C^*$.\\
(iii) If $w\in V$ is a sink then  
$\{p^* \> \vert \> p \in Path \Gamma ,\> tp=w \}$ is a ${\bf k}$-basis  for the $L_{\bf k}(\Gamma)$-module $wL_{\bf k}(\Gamma)$ and $wL_{\bf k}(\Gamma)$ is simple if and only if ${\bf k}$ is a field. 
\end{proposition}

\begin{proof}
(i) This is a graded homomorphism with respect to the universal grading by the free group on $E$ as described above. Hence the kernel is a graded ideal. Every homogeneous element of ${\bf k}\Gamma$ is of the form $\lambda p$ with $\lambda \in{\bf k}$ and $p \in Path \Gamma$. Since $\lambda p=0$ in $L_{\bf k}(\Gamma)$ only if $\lambda =0$ by Proposition \ref{pp^*}(ii), the claim follows.\\

(ii) Note that $v$ is the multiplicative identity of the corner algebra $vL_{\bf k}(\Gamma)v$. Paths $p$, $q$ in $ \Gamma$ with $sp=v=sq$ can not exit $C$. So, if $tp=tq$ and $p$ and $q$ have positive length we have $p=p_1e$ and $q=q_1e$ for some arrow $e$ on $C$. Thus $pq^*=p_1q_1^*$ since $se=ee^*$ for all $e$ on $C$ by (CK2). Repeating this we see that $vL_{\bf k}(\Gamma)v$ is spanned by $C^n$ where $C^0:=v$ and $C^{-1}:= C^*$ by Fact \ref{fact}(iii). We define an epimorphism from ${\bf k}[x,x^{-1}]$ to $vL_{\bf k}(\Gamma)v$ sending $x$ to $C^*$.
This epimorphism is one-to-one because $\{C^n\> \vert \> n \in \mathbb{Z}\}$ is linearly independent: A finite subset of  $\{C^n\> \vert \> n \in \mathbb{Z}\}$ is mapped to a set of distinct paths in $\Gamma$ after multiplying by a sufficiently high power of $C$, which are linearly independent by (i).\\

(iii) Since $w$ is a sink, $wL_{{\bf k}}(\Gamma)$ is spanned by  $\{p^* \> \vert \> p \in Path \Gamma ,\> tp=w \}$. Applying the anti-automorphism $*$ to this set we get a linearly independent set by (i). Hence  it is linearly independent and a basis  for $wL_{\bf k}(\Gamma)$. \\

If ${\bf k}$ is not a field then $\mathfrak{m}wL_{\bf k}(\Gamma)$ is a non-zero proper submodule of $wL_{\bf k}(\Gamma)$ where $\mathfrak{m}$ is a maximal ideal of ${\bf k}$, hence $wL_{\bf k}(\Gamma)$ is not simple.
If ${\bf k}$ is a field and $0\neq M$ is a submodule of $wL$ then  there is an $m=\sum \lambda_i p_i^* \in M$ with $\lambda_1\neq 0$ and $p_i$ distinct. Now $m(\frac{1}{\lambda_1}p_1)=w \in M$ and $w$ generates $wL_{{\bf k}}(\Gamma)$. Hence $M=wL_{{\bf k}}(\Gamma)$ showing that $wL_{{\bf k}}(\Gamma)$ is simple.
\end{proof}\\

The following Proposition is useful in showing that certain epimorphisms are isomorphisms. When ${\bf k}$ is a field, the last claim about graded homomorphisms is known as the Graded Uniqueness Theorem \cite[Theorem 2.2.15]{aam17}.\\

\begin{proposition} \label{bire-bir}
If $I$ and $J$ are right (respectively, left) ideals of $L_{\bf k}(\Gamma)$ then $I=J$ if and only if $I\cap {\bf k}\Gamma= J \cap  {\bf k}\Gamma $ (respectively, $I \cap {\bf k}\Gamma^*= J \cap  {\bf k}\Gamma^*$). In particular, $I=0$ if and only if  $I\cap {\bf k}\Gamma =0$ (respectively, $I\cap {\bf k}\Gamma^* =0$). Hence a ${\bf k}$-algebra homomorphism $\varphi$ from $L_{\bf k}(\Gamma)$ to a ${\bf k}$-algebra is one-to-one if and only if the restriction of $\varphi$ to ${\bf k}\Gamma$ is one-to-one. If $\varphi$ is a $\mathbb{Z}$-graded homomorphism then $\varphi$ is one-to-one if and only if the restriction of $\varphi$  to ${\bf k}v$ is one-to-one for all $v \in V$.
\end{proposition}

\begin{proof}
If $I \neq J$ then we may assume that there is an $\alpha = \sum_{i=1}^n \lambda_i p_iq_i^* \in I\setminus J $. Since $\alpha \sum v =\alpha $  where the sum is over $\{ v=sq_i \> \vert \> 1\leq i \leq n \}$ there is a vertex $w$ with $\alpha w \in I \setminus J$. If $w$ is a sink then $\alpha w \in ({\bf k}\Gamma \cap I)\setminus ({\bf k}\Gamma \cap J)$ and we are done. If $w$ is not a sink then $\alpha \sum_{se=w} ee^*=\alpha w \in I\setminus J$, so there is $e\in E$ with $\alpha e \in I\setminus J$. This shortens the $q_i^*$s and, repeating this we end up with an element in $(I\cap {\bf k}\Gamma ) \setminus (J\cap {\bf k}\Gamma )$. Thus $I\cap {\bf k}\Gamma \neq J\cap {\bf k}\Gamma $. If $I \neq J$ are left ideals of $L_{\bf k}(\Gamma)$ then $I^*\neq J^*$ are right ideals and we can apply the discussion above to get 
$I^*\cap {\bf k}\Gamma  \neq J^*\cap {\bf k}\Gamma $, hence 
$I\cap {\bf k}\Gamma^*  =(I^*\cap {\bf k}\Gamma)^* \neq (J^*\cap {\bf k}\Gamma)^*=J\cap {\bf k}\Gamma^* $. \\

If $\varphi$ is a one-to-one homomorphism from $L_{\bf k}(\Gamma)$ to a ${\bf k}$-algebra then every restriction of $\varphi$ is also one-to-one. Conversely, if $\varphi$ is not one-to-one then $0\neq Ker \varphi$. Hence $Ker \varphi \cap {\bf k}\Gamma \neq 0$ and the restriction of $\varphi$ to ${\bf k}\Gamma$ is not one-to-one. If $\varphi$ is a graded homomorphism then $0\neq Ker \varphi$ is a graded ideal, so we can find $0\neq \sum \lambda_i p_i \in Ker \varphi \cap {\bf k}\Gamma$ with $p_i$ distinct, $\lambda_i \neq 0$ for each $i$ and all paths $p_i$ having the same length.
Now $p_1^*\sum \lambda_ip_i =\lambda_1 sp_1 \in {\bf k}sp_1 \cap Ker \varphi$ because $p_1^*p_i=0$ for $i=2,3, \cdots , n$
by Fact \ref{fact}(i), since $p_i$s have the same length.
\end{proof}\\

We want to show that most modules obtained from unital modules using standard constructions, such as taking quotients, submodules, sums and extensions,  are also unital.

\begin{lemma} \label{unital}
 If $M$ is an $L_{\bf k}(\Gamma)$-module then the following are equivalent:\\
 (i) $M$ is unital.\\
 (ii) If $0\neq m \in M$ then $0< \vert V_m \vert < \infty $ where $V_m:= \{ v \in V \> \vert \> mv\neq 0\}$. \\
 (iii) $M=\Sigma_{v \in V} Mv= \oplus_{v \in V} Mv$.
 \end{lemma}
 
 \begin{proof}
 (i) $\Leftrightarrow$ (iii): If $M$ is unital then $m=\Sigma m_ia_i$ for all $m \in M$ where $a_i=\Sigma \lambda_{ij}p_{ij}q_{ij}^*$ for finitely many $\lambda_{ij} \in {\bf k}$ and $p_{ij},\> q_{ij}$ in $Path \Gamma$. Rearranging $\Sigma m_i a_i$ by grouping terms with the same $sq_{ij}$ we see that $m \in \Sigma_{v \in V} Mv$. Since $V$ is a set of orthogonal idempotents, $\Sigma Mv=\oplus Mv$.\\
 
 Conversely, if $M=\Sigma_{v \in V} Mv $ then for all $m \in M$ we have $m = \Sigma_{i=1}^n m_i$ where $m_i \in Mv_i$ for some $v_1, \> v_2, \> \cdots , \> v_n$ by (iii). Since $m_i=m_iv_i$ and $v_i \in L_{\bf k}(\Gamma)$ we get $m = \Sigma m_iv_i$, so $M= ML_{\bf k}(\Gamma)$, that is $M$ is unital.\\

 (ii) $\Leftrightarrow $ (iii): Assuming (ii), if $m \in M$ then $\Sigma_{v \in V} mv$ is actually a finite sum. Also $(m-\Sigma mv)u=mu-mu= 0$ by the relations (V) for all $u \in V$. Now  $m-\Sigma mv=0$ by (ii), hence $m \in \Sigma Mv$. \\
 
 Conversely, if  $M=\Sigma Mv$ then $m=\Sigma_{i=1}^n m_i v_i$ as above, hence $V_m \subseteq \{v_1, \> v_2, \>  \cdots , v_n\}$, thus finite. Also, if $mv=0$ for all $v \in V$ then $m v_i = m_i v_i =0 $ for $i=1,2, \cdots,n$ and hence $m=0$. So, if $0\neq m $ then $V_m \neq \emptyset$. 
  \end{proof}
  
  \begin{proposition} \label{unital2}
  The full subcategory of unital $L_{\bf k}(\Gamma)$-modules is a Serre subcategory of $Mod_{L_{\bf k}(\Gamma)}$ which is closed under colimits. 
  \end{proposition}
  \begin{proof}
  Clearly, the 0 module is unital. If $M$ is unital, i.e., $M=MR$ then any quotient $N$ of $M$ also satisfies $N=NR$. 
  Using Lemma \ref{unital}(ii) we see that unital modules are closed under taking submodules. Being unital is clearly invariant under isomorphisms. If $A \hookrightarrow M \longrightarrow M/A$ is a short exact sequence of $L(\Gamma)$-modules with $A$ and $M/A$ unital then $V_m=V_{m+A} \cup V_{a}$ with $a=m-\sum mv$ where the sum is over $v \in V_{m+A}$ for all $m \in M$. Since $m-\sum mv $ is in $A$ and for $m\neq 0$ at least one of $V_{m+A}$ or $V_a$ is nonempty, $M$ is unital by Lemma \ref{unital}(ii). This proves that the subcategory of unital $L(\Gamma)$-modules is a Serre subcategory. \\
  
  If $M=\oplus M_i$ with $M_i$ unital for all $i$ then $M_i= \oplus_{v \in V} M_iv$ for each $i$ by Lemma \ref{unital}(iii). Also $Mv=\{ m \in M \> \vert \>  mv =m \}=\oplus M_iv$. Changing the order of summation we see that $M=\oplus_{v \in V} Mv$. Hence $M$ is unital by Lemma \ref{unital}(iii), so an arbitrary direct sum of unital $L(\Gamma)$-modules is also unital. Since any colimit is a quotient of a direct sum, we are done. 
  \end{proof} \\

 Since $L_{\bf k}(\Gamma)$ regarded as a free $L_{\bf k}(\Gamma)$-module is unital, projective modules $vL_{\bf k}(\Gamma)$ with $v \in V$ and their direct sums are all unital by Proposition \ref{unital2}, as well as their quotients. However, the subcategory of unital modules is not closed under arbitrary products. For instance, if $\Gamma$ has infinitely many vertices $v_0, v_1, v_2, \cdots $ then $m \in M=L_{\bf k}(\Gamma)^{\mathbb{N}}$ with $m(i)=v_i$ for $i \in \mathbb{N}$ violates Lemma \ref{unital}(ii), hence $M$ is not unital. \\

   \medskip
\subsection{The Reduction Algorithm }
\medskip

 This section is about the consequences of a geometric (graph theoretic)  process we call the \textbf{reduction algorithm} (\cite{ko1}, \cite{ko2}) defined on a row-finite digraph $\Gamma=(V,E)$: For a loopless nonsink $v \in V$, we replace each path $fg$ of length 2 such that $tf=v=sg$ with an arrow labeled $fg$ from $sf$ to $tg$ and delete $v$ and all arrows touching $v$. (Note that $fg$ denotes a path in $\Gamma$, but an arrow in its reduction.) In particular, if $v$ is a source but not a sink, then we delete $v$ and all arrows starting at $v$ without adding any new arrows. We may repeat this  as long as there is a loopless non-sink. Any digraph obtained during this process is called a \textbf{reduction} of $\Gamma$. 
If $\Gamma$ is finite, after finitely many steps we will reach a \textbf{complete reduction} of $\Gamma$,  which has no loopless nonsinks. A digraph in which every vertex is either a sink or has a loop, is called \textbf{completely reduced}.\\

In the example below, $\Gamma_1$, $\Gamma_2$ and $\Gamma_3$ are reductions of the digraph $\Gamma$. The number of arrows from one vertex to another is indicated by the number above the arrow (so, in $\Gamma$ there are 3 arrows from $v$ to $w$).  $\Gamma_3$ is a complete  reduction of $\Gamma$.

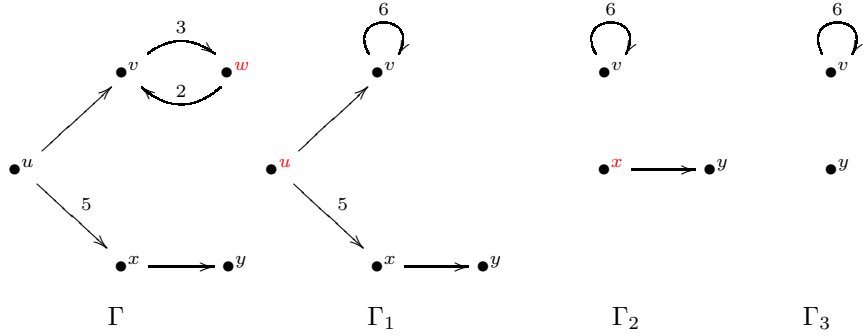
\begin{example}

$$ \xymatrix{&   {\bullet}^{v} \ar@/^1pc/[r]^{3} & {\bullet}^{\text{\textcolor{red}{w}}} \ar@/^1pc/[l]_{2} \\ {\bullet}^{u} \ar@{->}[ur] \ar@{->}[dr]^{5} & \\
                & \bullet^{x} \ar@{->}[r]  & \bullet^{y}  }
 \xymatrix{&{\bullet}^{v}\ar@(ul,ur)^{6} & \\ {\bullet}^{\text{\textcolor{red}{u}}} \ar@{->}[ur] \ar@{->}[dr]^{5}& \\
               & \bullet^{x} \ar@{->}[r]  & \bullet^{y}  }
 \xymatrix{&{\bullet}^{v}\ar@(ul,ur)^{6} & \\ 
                & \bullet^{\textcolor{red}{x} } \ar@{->}[r]  & \bullet^{y} }
 \xymatrix{& {\bullet}^{v} \ar@(ul,ur)^{6}&\\ 
            &  \bullet^{y}  }
$$
 $$\quad \Gamma \qquad\qquad  \qquad  \quad  \qquad  \Gamma_1\qquad  \qquad \qquad\qquad  \Gamma_2\qquad  \qquad \qquad \Gamma_3 $$\\
\end{example}

The complete reduction $\Gamma_3$ of the digraph $\Gamma$ above does not depend on the choice of reductions. However, complete reductions are not unique (up to digraph isomorphism) in general. For instance, the digraph $\Lambda$ below has two non-isomorphic complete reductions:\\
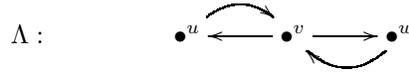
$$\Lambda : \  \quad  \xymatrix{& {\bullet}^{u} \ar@/^1pc/[r] &\ar [l]  {\bullet}^{v}  \ar [r]  & {\bullet}^{w} \ar@/^1pc/[l] }$$\\

\noindent
The complete reductions of $\Lambda$:\\
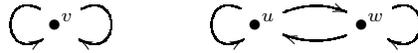
$$   \xymatrix{ & {\bullet}^{v} \ar@(ul,dl) \ar@(ur,dr)  } \qquad \quad 
 \xymatrix{& {\bullet}^{u}\ar@(ul,dl) \ar@/^0.6pc/[r] & {\bullet}^{w} \ar@(ur,dr)    \ar@/^0.5pc/[l] }\quad $$ \\

A digraph and all reductions of it have the same set of sinks. Cycles may get shorter under each reduction and but they do not disappear. If $\Gamma$ is finite and the cycles of $\Gamma$ are pairwise disjoint then $\Gamma$ has a unique complete reduction up to isomorphism. All cycles in $\Gamma$ become loops in the complete reduction. The vertices of the complete reduction  correspond to the sinks and the cycles of $\Gamma$. 

\begin{theorem} \cite[Theorem 4.1]{ko2} \label{Reduced}
Let ${\bf k}$ be a commutative ring with 1. If  $\Gamma'$ is a  reduction of $\Gamma$, then $L_{\bf k}(\Gamma)$ and $L_{\bf k}(\Gamma')$ are Morita equivalent, that is, their unital module categories are equivalent. This equivalence preserves the subcategories of finite-dimensional modules.
\end{theorem}

   The proof of Theorem \ref{Reduced} when ${\bf k}$ is a field is given in 
 \cite[Theorem 4.1]{ko2}. Similar to Propositon \ref{teorem} (which is the main tool) this proof also works over a commutative ring with 1. \\

From the quiver representation viewpoint, reduction corresponds to restricting the representation to the remaining vertices. The new arrow $ef$ is assigned the composition of the ${\bf k}$-module homomorphisms assigned to $e$ and $f$.  We recover the original representation $\rho$ by assigning $\oplus_{se=v} \rho(te)$ to the deleted vertex $v$. The details are given in \cite[Theorem 4.1]{ko2}.\\

   When $\Gamma$ is finite, a reduction $\Lambda$ of $\Gamma$ is isomorphic to the corner algebra $(\sum u)L_{\bf k}(\Gamma) (\sum u) $ where the sum is over the vertices of $\Lambda$, i.e., all the vertices that were not eliminated during the reduction process. Actually $(\sum u)L_{\bf k}(\Gamma) (\sum u) $ makes sense even if $\Gamma$ is infinite because the product of all but finitely many vertices with an element of $L_{\bf k}(\Gamma)$ is 0. 
   
   \begin{proposition} \label{corner1}
   If $\Lambda$ is a reduction of $\Gamma$ then 
   $L_{\bf k}(\Lambda)\cong (\sum u)L_{\bf k}(\Gamma) (\sum u)$ where the sum is over the vertices of $\Lambda$ as above.
   \end{proposition}
   
   \begin{proof}
   It suffices to prove this for a 1-step reduction eliminating the vertex $v$. We define a $*$-algebra homomorphism $\varphi$ from $L_{\bf k}(\Lambda)$ to $L_{\bf k}(\Gamma)$ by sending each vertex and each original arrow to itself and a new arrow $ef$ of $\Lambda$ to the path $ef$ in $\Gamma$. It's routine to check that the relations are satisfied, checking (CK2)  uses (CK2) of $\Gamma$ twice if new arrows are involved.\\
   
   If $pq^*\in L_{\bf k}(\Gamma)$ with $sp \neq v\neq sq$, but $tp=v=tq$ then $pq^*=pvq^*=p(\sum_{sf=v} ff^*)q^*$ since $v$ is not a sink. Note that $p$ and $q$ are paths of positive length because $sp\neq v =tp$ and $sq\neq v =tq$, hence $p=p'e$ and $q=q'g$ with $ef$ and $gf$ being images of new arrows in $\Lambda$ for all $f$ with $sf=v$. Similarly, if $p'$ or $q'$ pass through $v$ then we have a path of length 2 corresponding to a new arrow in $\Lambda$, using $tp'=se\neq v \neq sg=tq'$ since $v$ is loopless. Thus image of $\varphi$ is $(\sum u)L_{\bf k}(\Gamma) (\sum u)$.\\

   The restriction of $\varphi$ to ${\bf k}\Lambda$ is one-to-one, hence $\varphi$ is an isomorphism onto $(\sum u)L_{\bf k}(\Gamma) (\sum u)$ by Propositon \ref{bire-bir}. 
  \end{proof}
       
  \begin{corollary}
 Let $\Gamma$ be a finite digraph with pairwise disjoint cycles and let $\Lambda$ be the complete reduction of $\Gamma$. If $U$ is a set containing all the sinks in $\Gamma$ and exactly one vertex from each cycle in $\Gamma$ then $$L_{\bf k}(\Lambda) \> \cong \>  \big(\sum_{u \in U} u\big)L_{\bf k}(\Gamma) \big(\sum_{u \in U} u\big) .$$
  \end{corollary}     
       
   \begin{proof}
   A complete reduction of a finite digraph whose cycles are pairwise disjoint is gotten by eliminating all the vertices other than the sinks and one vertex from each cycle. Applying Proposition \ref{corner1} yields the result.
   \end{proof}

 \subsection{Normalization/Pincushion}
 A variation/generalization of the reduction algorithm is {\bf normalization} which gives a $*$-algebra isomorphism (not just a Morita equivalence)  between $L(\Gamma)$ and  $L(\Gamma')$ where $\Gamma'$ is a normalization of the digraph $\Gamma$. \\

      Let $\Gamma =(V,E)$ be an arbitrary digraph and $v$ be a loopless regular vertex in $\Gamma$. The 1-step normalization $\Gamma_v = (V_v, E_v)$ of $\Gamma$ is the digraph where  $$ V_v :=(V \setminus \{ v\})\cup \{v_e \> \vert \> se=v\} \qquad \quad $$  $$E_v:=(E \setminus  t^{-1}
      (v)) \cup \{ ef \>\vert te=v=sf \} $$ 
      with $s(ef):=se$, $t(ef):=tf$, $\> se:=v_e$ for $e \in s^{-1}(v)$. The source function is the same as that of $\Gamma$ on $E \setminus (s^{-1}(v) \cup t^{-1}(v))$; the target function is the same as that of $\Gamma$ on $E \setminus t^{-1}(v)$.\\

Deleting the vertices $v_e$ and arrows $e $ with $se=v_e$ from $\Gamma_v$ we obtain $\Gamma_v'$, a 1-step reduction of $\Gamma$. We refer to the arrows $e$ with $se=v_e$ in $\Gamma_v$ as \textit{pins}. We obtain the (\textit{pincushion}) $\Gamma_v$ from the 1-step reduction $\Gamma_v'$ by sticking the pins $e \in s^{-1}(v)$ to $te$ in $\Gamma_v'$.\\

\begin{theorem}
    $L(\Gamma)$ is $*$-isomorphic to $L(\Gamma_v)$.
\end{theorem}
\begin{proof}
    Let $\varphi :V_v \sqcup E_v \longrightarrow L(\Gamma)$ be given by   
$$\varphi (u):=u  \> \text{ for }\>  u \in V\setminus \{ v\} \> \text{ and } \>  \varphi (v_e):=ee^*$$  $$\varphi(f):=f \> \text{ for } \> f \in E \setminus t^{-1}
      (v)\> \text{ and } \> \> \varphi(ef) : = ef \in L\Gamma.$$ 
      \noindent
      Note that $ef$ denotes a path of length 2 in $\Gamma$, but an arrow in $\Gamma_v$, which one it is should be clear from the context.
It's routine to check that the defining relations of $L(\Gamma_v)$ are satisfied, so $\varphi$ extends to a $*$-algebra homomorphism. \\

Let $\psi :V \sqcup E \longrightarrow L(\Gamma_v)$ be given by   
$$\psi (u):=u  \> \text{ for }\>  u \in V\setminus \{ v\} \> \text{ and } \>  \psi (v):= \sum_{se=v} v_e $$  $$\psi(f):=f \> \text{ for } \> f  \notin t^{-1}
      (v)\>, \> \> \psi(e) : = \sum_{sf=v} (ef)f^*  \text{ for } e \in t^{-1}(v).$$ 
It is routine to check that the relations of $L(\Gamma) $ are satisfied. Since $\varphi \circ \psi =id_{L(\Gamma)}$ and $\psi \circ \varphi =id_{L(\Gamma_v)}$, $L(\Gamma)$ and $L(\Gamma_v)$ are $*$-isomorphism.
\end{proof}\\

We say that $\Gamma$ is in {\bf normal form} if each regular vertex of $\Gamma$ is either a source emanating a single arrow or the base of a loop. If $\Gamma$ has a finite number of regular vertices then after applying (finitely many, 1-step) normalization moves we get to a $\Gamma'$ in normal form (with $L(\Gamma) \cong L(\Gamma')$).\\

Note that the $*$-algebra isomorphism $\varphi$ between $L(\Gamma)$ and $L(\Gamma_v)$ is not graded with respect to the standard $\mathbb{Z}$-gradings on $L(\Gamma)$ and $L(\Gamma_v)$.

\subsection{Jellyfish}
       
 Let $\Gamma$ be a row-finite digraph, $H\subseteq V$ hereditary and saturated, $(H)$ ideal generated by $H$ and 
 $P_H =\{ p \in Path \Gamma \> \vert \> p=p'f, \> f \in E, \> tf \in H, \> sf \notin H\}$.\\

\noindent
$\Gamma_{(H)} := (H \sqcup V_{(H)}, \> E_H \sqcup E_{(H)}, \> s,\> t)$ is called the \textit{jellyfish} of $H$ where 
$$E_H :=s^{-1}(H) \cap t^{-1}(H), \> \> 
V_{(H)} := \{v_p \vert \> p \in P_H\} \> \> \textit{ and } \> \> E_{(H)}:=\{ e_p \vert \> p \in P_H \}$$ such that    
$$\textit{ 
   for all } e \in E_H \qquad 
se := s e, \> \> te:= t e \>\> $$ 
   $$ \textit{ for all } e_p \in E_{(H)} \qquad   se_p := v_p\qquad \quad
   $$
   $$te_p:= v_q \quad \textit{ if } l(p)>1 \textit{  
   and } p=eq  $$
   $$te_p:= tp \quad \textit{ if } l(p)=1 \qquad \qquad \quad $$
   
\begin{theorem}
    $L(\Gamma_{(H)})$ is graded isomorphic to  $(H)$ via $v \leftrightarrow v $ if $v\in H$, for all $p \in P_H$, $v_p \leftrightarrow  pp^*$ and  $e_p \leftrightarrow eqq^*$
where $p=eq$ with $e\in E$. 

\end{theorem}   
   \section{Gelfand-Kirillov Dimension}    
       
 \subsection{A Basis for $L_{\bf k}(\Gamma)$ when $\Gamma$ is a finite digraph whose cycles are  pairwise disjoint }

 When the cycles of $\Gamma$ are pairwise disjoint, the preorder $\leadsto$ defines a partial order on the set of sinks and cycles in $\Gamma$. For each cycle $C$ let's fix a vertex $v_C$ on $C$ and declare that $sC=v_C$.

\begin{theorem} \label{basis}
If $\Gamma$ is a finite digraph whose cycles are pairwise disjoint then $\mathcal{B} :=\{ pq^*\> \vert \> tp=tq \textit{ is a sink } \} \cup \{ pC^nq^* \> \vert \> p, q \in P_{v_C} ,\> n \in \mathbb{Z},\>  C \textit{ is a cycle   } \}$ is a ${\bf k}$-basis for $L_{\bf k}(\Gamma)$ where $C^0:=v_C$  and $C^{-n}:=(C^*)^n$ for $n>0$.
\end{theorem}

\begin{proof} Let $v_C:=sC$. we know that $\{ab^* \>  \vert \> a, \> b \in Path(\Gamma), \> ta=tb \>\}$ spans $L(\Gamma)$. If $ta=tb$ is neither a sink nor $v_C$ for some cycle $C$ then we apply (CK2) to get 
$$ab^*= \sum_{se=ta} aee^*b^* .$$

We repeat this process as needed, until all the terms are of the form $pq^*$ with $tp=tq$ either a sink or $v_C$ for some cycle $C$. The terms with $tp=v_C$ for some cycle $C$ are not modified, so no $v_C$ is repeated. Hence this process terminates since there are no infinite paths in $\Gamma$ not containing cycles. Also we do not create any occurrences of $CC^*$.\\

Starting with the relation (CK2) for $v=v_C$ and applying the process above to the right-hand-side, we get $CC^*=v_C-\sum pp^*$ with no $p$ containing a cycle and every $tp$ being  a sink or $v_D$ for some cycle $D\neq C$. (Here it is essential that the cycles of $\Gamma$ are pairwise disjoint.) Using this equation we eliminate occurrences of $C^n(C^*)^n$ for some positive integer $n$. Hence $\mathcal{B}$ spans $L(\Gamma)$.\\

To see that $\mathcal{B}$ is linearly independent over ${\bf k}$ we assume to the contrary that $\sum_{i=1}^n \lambda_i p_i q_i^*=0$ with all $\lambda_i \neq 0$ and $p_i q_i^* $ are distinct elements of $\mathcal{B}$. Using the universal grading on $L(\Gamma)$ we may assume that $p_i q_i^*$ for all $i$ have the same grade, that is, there are  $a$, $b$ and $r_i$ in $Path \Gamma$ with $r_i r_i^*$ in $\mathcal{B}$ such that $p_i q_i^*=a r_i r_i^* b^*$ for all $i$. Now $a^*(\sum \lambda_i p_i q_i^*)b= \sum \lambda_i r_i r_i^*=0$. If $tr_i$ is a sink for all $i$ then $r_1^* (\sum \lambda_i r_i r_i^*)r_1= \lambda_1 tr_1=0$ since $r_1$ can not be an initial segment of $r_i$ and $r_i$ can not be an initial segment of $r_1$  for $i\neq 1$.\\

If $tr_i$ is not a sink for some $i$ then,  reordering if necessary, we may assume that $tr_1=v_C$ with $tr_1$ maximal among $tr_i$ and $l(r_1)$ minimum among those $r_i$ with $tr_i$ maximal. Pick $n>l(r_i)$ for all $i$. Then $(r_1C^n)^*(\sum \lambda_i r_ir_i^*)r_1C^n= \lambda_1 tr_1=0$, because $r_1C^n$ is too long to be an initial segment of $r_i$ for $i\neq 1$ and  $r_i$ can not be an initial segment of $r_1C^n$: Otherwise (i) If $tr_i=tr_1$ then $r_i=C^m $ since the cycles of $\Gamma$ are pairwise disjoint. Moreover,  $m>0$ because $r_i=v_C$ would give $r_1=v_C=r_i$ by the minimality of $r_1$. But $r_i=C^m$ with $m>0$  contradicts $r_ir_i^* \in \mathcal{B}$. (ii) If $tr_i\neq tr_1$ then $r_i$ can not be an initial segment of $C^nr_1$ since $tr_i$ can not connect to $r_1$. In both cases we have $\lambda_1 v_C = 0$, a contradiction since $\lambda_1 \neq 0$.
  \end{proof}\\

\begin{remark}
We stated Theorem \ref{basis} for finite digraphs because it does not hold without some finiteness hypothesis. For instance, when the digraph has no sinks or cycles, (like the digraph $\Gamma$ below) the set $\mathcal{B}$ defined above is empty.

$$\xymatrix{ {\bullet}^0 
 \ar@{->}[r]   &{\bullet}^1 
 \ar@{->}[r]    &  \bullet^2
 \ar@{->}[r]    &  \bullet^3
  \cdots}$$

\noindent
A reasonable finiteness condition, more general than $\Gamma $ being finite while sufficient for the Theorem \ref{basis} to hold, is:
$$\textit{Every infinite path in $\Gamma$ digraph meets a cycle.}$$ 

\end{remark}

\begin{remark} \label{baz2}
We also get a basis $v\mathcal{B}:=\{ pq^* \in \mathcal{B} \>\vert \> sp=v \} $   of the projective $L(\Gamma)$-module $vL(\Gamma)$ for all $v \in V$: We see that  $v\mathcal{B}$ spans $vL(\Gamma)$ as in the proof of Theorem \ref{basis} and $v\mathcal{B}$ is  linearly independent since it is a subset of $\mathcal{B}$. Similarly, $$\bigcup_{v \in X} v\mathcal{B} \textit{ is a basis for the projective } L(\Gamma)-\textit{module } \big( \sum_{v \in X} v \big) L(\Gamma) =\bigoplus_{v \in X} vL(\Gamma)$$
where $X$ is a finite subset of $V$.
\end{remark}

Let $P_w:=\{ q \in Path \Gamma \> \vert \> \> tq =w \}$ where $w$ is a sink and  $P_v :=\{q \in Path \Gamma \> \vert \> \> tq =v, \> q\neq pC \}$ where $C$ is a cycle in $\Gamma$ with $v=sC$.

\begin{corollary} \label{corner}
If $C$ is a cycle in $\Gamma$ with no exit and  $v=sC$ then
 $vL_{{\bf k}}(\Gamma)$ is a free left $vL_{{\bf k}}(\Gamma)  v$-module with basis $P_v^*:=\{q^*\vert \> q\in P_v\}$.
\end{corollary}
\begin{proof}
 The basis $\mathcal{B}$ restricted to $vL_{{\bf k}}(\Gamma)$ is $v\mathcal{B}= \{ C^nq^* \vert \> n \in \mathbb{Z} ,\> q \in P_v \}$ and $vL_{{\bf k}}(\Gamma)v \cong {\bf k}[C,C^*]$ by Proposition \ref{Path}(ii), proving the claim.
\end{proof}\\

As remarked before, when there is no confusion we will use the abbreviation $L$ for $L_{\bf k}(\Gamma)$.

\begin{lemma} \label{Temel}
(i) Let $w$ be a sink. If the $L$-module $M$ is generated by $Mw$ then $M=\bigoplus_{q\in P_w} Mwq^*$ as a ${\bf k}$-module and $Mw \otimes_{vLv} vL \cong M$  as right $L$-modules.\\
(ii) Let $C$ be a cycle in $\Gamma$ with no exit, $v:=sC$ and $P_C :=\{q \in Path \Gamma \> \vert \> \> tq =v, \> q\neq pC \}$. If the $L$-module $M$ is generated by $Mv$ then $M=\bigoplus_{q\in P_C} Mvq^*$ as ${\bf k}$-modules and $Mv \otimes_{vLv} vL \cong M$  as right $L$-modules.
\end{lemma}

\begin{proof}
(i) This proof is omitted since it is very similar to and somewhat simpler than the proof of (ii) below.\\

(ii) If $m \in M$ then $m=\sum_j (m_j'\sum_i \lambda_i p_ir_i^*)$ with $m_j' \in Mv$ and $p_ir_i^* \in \mathcal{B}$. Since $C$ has no exit, $tp_i=v=tr_i$. Hence $m=\sum_k m_kq_k^*$ with $q_k \in P_C$ for all $k$, so $M=\sum_{q\in P_C} Mvq^* $. The sum is direct because $q^*q'=0$ if $q \neq q'$ in $P_C$ since $q$ can not be an initial segment of $q'$ and $q'$ can not be an initial segment of $q$.\\
We have an epimorphism $$Mv \bigotimes_{vLv} vL \longrightarrow  M=\bigoplus_{q\in P_C} Mvq^*$$ given by $m\otimes a \mapsto ma$. Since $$Mv \bigotimes_{vLv} vL= \sum_{q \in P_C} (Mv \bigotimes q^*)$$ 
this sum is also direct (it maps onto a corresponding direct sum) and 
 we get the stated isomorphism. 
 \end{proof}\\

\begin{proposition} \label{Ser}
Let $C$ be a cycle in $\Gamma$ with no exit and $v:=sC$. The full subcategory $\mathcal{S}_v$ with  objects M generated by $Mv$ of $\textbf{Mod}_{L_{\bf k}(\Gamma)}$ is a Serre subcategory.
\end{proposition}

\begin{proof}
 Clearly, the $0$ module is in $\mathcal{S}_v$  and if $M$ is in $\mathcal{S}_v$, then any module isomorphic to $M$ is also in $\mathcal{S}_v$. It is routine to see that $\mathcal{S}_v$ is closed under quotients and extensions. Hence all we need to show is that $\mathcal{S}_v$ is closed under taking submodules. Let $M$ be in $\mathcal{S}_v$ and $N$ be a submodule of $M$. If $m \in N \subset M$ then $m =\sum_{i=1}^n m_iq_i^*$ with $m_i \in Mv$ and $q_i \in P_C$ for all $i$ by Lemma \ref{Temel}. Since $q_i^*q_j=0$ for $i\neq j$ we have $mq_j=m_j \in Nv$. It follows that $N$ is generated by $Nv$. 
\end{proof}

\begin{theorem} \label{Ser2}
If $v$ is a sink in $\Gamma$ or $v=sC$ where $C$ is a cycle in $\Gamma$ with no exit then the full subcategory $\mathcal{S}_v$ of $L$-modules $M$ that are generated by $Mv$ is equivalent to the category of $vLv$-modules via the functors
%\textcolor{red}{${\bf k}[x,x^{-1}]\cong vLv$-modules via the functors}
$$\underline{\>\>\>}\bigotimes_{vLv} vL : \textbf{Mod}_{vLv} \longrightarrow \textbf{Mod}_L$$  $$Hom^L (vL, \underline{\>\>\>}) :\textbf{Mod}_L \longrightarrow \textbf{Mod}_{vLv} \> .$$  
Moreover, an $L$-module $M$ generated by $Mv$ is simple if and only if $Mv$ is a simple $vLv$-module and also $M$ is projective if and only if $Mv$ is a projective $vLv$-module.
\end{theorem}

\begin{proof}
Since $vL$ is a $(vLv, L)$-bimodule $\underline{\>\>\>}\bigotimes_{vLv} vL $ is a functor from $\textbf{Mod}_{vLv}$ to $\textbf{Mod}_L$. Its right adjoint $Hom^L(vL, \underline{\>\>\>})$ is naturally isomorphic to the functor sending $M$ to $Mv$ and $\varphi :M \longrightarrow N$ to $\varphi \vert_{Mv}: Mv \longrightarrow Nv$. Both are exact functors because $vL$ is a free left $vLv$-module by Corollary \ref{corner} and $vL$ is a projective  $L$-module. If $v=sC$ and $A$ is a $vLv \cong {\bf k}[x,x^{-1}]$-module then $A\otimes vL =\oplus_{q\in P_v} (A\otimes q^*)$. The unique $q\in P_C$ with $sq=v$ is $v$. Hence $(A\otimes vL)v$ is naturally isomorphic to $A$, so the composition $Hom^L (vL,\> \underline{\>\>\>}\bigotimes_{vLv} vL)$ is naturally isomorphic to the identity functor. Since $Mv \bigotimes vL \cong M$ by Lemma \ref{Temel} when $M$ is generated by $Mv$, the functor  $\underline{\>\>\>}\bigotimes vL $ maps  $\textbf{Mod}_{vLv}$ into $\mathcal{S}_v$. Also $Hom^L (vL, \underline{\>\>\>})$ restricted to $\mathcal{S}_v$ gives an equivalence with the category $\textbf{Mod}_{vLv}$ by Proposition \ref{Ser}. The proof when $v$ is a sink is similar but easier.\\

$M$ in $\mathcal{S}_v$ is simple as an $L$-module if and only if $M$ is simple in $\mathcal{S}_v$ since all submodules of $M$ are also in $\mathcal{S}_v$. Since $Hom^L (vL, \underline{\>\>\>})$ gives an equivalence between $\mathcal{S}_v$ and $\textbf{Mod}_{vLv}$, we see that $Mv\cong Hom^L (vL, M)$ is simple if and only if $M$ is simple.\\

If $M$ in $\mathcal{S}_v$ is projective as an $L$-module then it is clearly projective in the subcategory $\mathcal{S}_v$. If $M$ is a projective object in $\mathcal{S}_v$ and 
$$\begin{array}{cc}
     &M\\
     &\downarrow  \\
    A \rightarrow & B
\end{array}$$
is a diagram of $L$-modules with $A\longrightarrow B$ onto then $Av \longrightarrow Bv$ is also onto. We get a commutative diagram $$\begin{array}{ccc}
     &&M=MvL\\
   &\swarrow &\downarrow  \\
    AvL &\longrightarrow & BvL\\
    \hookdownarrow& &\hookdownarrow\\
    A&\longrightarrow &  B
\end{array}$$
since $M$ is a projective object in $\mathcal{S}_v$. Hence
$M$ is a projective $L$-module. Thus $M$ is projective $L$-module if and only if $M$ is a projective object in $\mathcal{S}_v$ if and only if $Mv$ is a projective $vLv$-module, because the functor $Hom^L (vL, \underline{\>\>\>})$ gives an equivalence between $\mathcal{S}_v$ and $\textbf{Mod}_{vLv}$ preserving projectivity. 
\end{proof}

\subsection{Gelfand-Kirillov Dimension of $L_{\bf k}(\Gamma)$}

Let $\mathbb{F}$ be a field, $A$ a finitely generated $\mathbb{F}$-algebra with 1. The Gelfand-Kirillov dimension of $A$ is:
 $$GK dim(A)=\limsup_{n \to \infty} \frac{log (dim(W^n))}{ log (n)}$$
 where $W$ is a finite dimensional $\mathbb{F}$-subspace generating $A$ with $1\in W$, and $W^n$ is the $\mathbb{F}$-span of  $n$-fold products of elements from $W$. The Gelfand-Kirillov dimension of $A$ is independent of the choice of $W$. The algebra $A$ has polynomial growth if and only if $GKdim(A)$ is finite.\\

The Leavitt path algebra $L_{\mathbb{F}}(\Gamma)$ is finitely generated if and only if $\Gamma$ is finite. In which case $L_{\mathbb{F}}(\Gamma)$ has $1=\sum_{v \in V} v$. If $\Gamma$ has intersecting cycles, say $C$ and $D$, with $sC=u=sD$ then the subalgebra generated by $C$ and $D$ is a free algebra in 2 noncommuting variables by Propositon \ref{Path}(i). Therefore $\mathbb{F}\langle C,D \rangle \>$ and hence $L_{\mathbb{F}}(\Gamma)$ have exponential growth, so $GKdim L_{\mathbb{F}}(\Gamma)$ is infinite. The converse is also true:

\begin{theorem}\cite[Theorem 5]{aajz12} \label{aajz12}
Let $\mathbb{F}$ be a field and $\Gamma$ be a finite digraph. The Gelfand-Kirillov dimenson of $L_{\mathbb{F}}(\Gamma)$ is finite if and only if the cycles of $\Gamma$ are pairwise disjoint.
\end{theorem}

We will prove a finer version of Theorem \ref{aajz12} which characterizes the Gelfand-Kirillov dimension of $L_{\mathbb{F}}(\Gamma)$ graph theoretically, via a \textit{height} function defined on the sinks and cycles of a finite digraph whose cycles are pairwise disjoint. \\

Let the cycles of $\Gamma$ be pairwise disjoint. Then the pre-order $\leadsto$ defines a partial order on the set of sinks and cycles in $\Gamma$. 
We define the {\bf height} function on the sinks and the cycles of a finite digraph $\Gamma$: The \textit{height of a sink} is 0. The \textit{height of a cycle with no exit} is 1. 
The \textit{height of a cycle} $C$ \textit{with an exit} is recursively defined as: $\textit{ht} (C) = 2+max \{ \textit{ht}( x)  \> \vert \>  C \leadsto x\> , \> C\neq x \} $. (This also defines the height of the vertices of the complete reduction since they are identified with the sinks and the cycles of $\Gamma$.) We define the height of $\Gamma$ to be the maximum of the heights of its cycles or 0 if $\Gamma$ has no cycles. \\

\begin{theorem} \label{height}
If $\mathbb{F}$ is a field and $\Gamma$ is a finite digraph whose cycles are pairwise disjoint then $GK dim L_{\mathbb{F}}(\Gamma)=ht(\Gamma)$. 
\end{theorem}
\begin{proof}
Let $W$ be the span of all the cycles of $\Gamma$, their duals, all paths not containing any cycles and their duals. Note that $1=\sum v \in W$ and $W$ generates $L_{\mathbb{F}}(\Gamma)$ since every element in the basis $\mathcal{B}$ is contained in $W^n$ for some $n$. \\

Let $C_1 \leadsto C_2 \leadsto \cdots \leadsto C_k$  and $C_{k+m} \leadsto C_{k+m-1} \leadsto \cdots \leadsto C_{k+1} \leadsto C_k$ be distinct cycles with $p_1$ a path containing no cycles from $sC_1$ to $sC_2$, similarly $p_2$ from $sC_2$ to $sC_3$, $\cdots , p_{k-1}$ from $sC_{k-1}$ to $sC_k$, $p_k$ from $sC_{k+1}$ to $sC_k$, $\cdots$, $p_{k+m-1}$ from $sC_{k+m}$ to $sC_{k+m-1}$. The elements in $W^n$ of the form $$sC_1^{n_0}C_1^{n_1}p_1C_2^{n_2}p_2 \cdots C_k^{n_k}p_k^*(C_{k+1}^*)^{n_{k+1}} \cdots p_{k+m-1}^* (C_{k+m}^*)^{n_{k+m}}$$  $$=C_1^{n_1}p_1C_2^{n_2}p_2 \cdots C_k^{n_k}p_k^*(C_{k+1}^*)^{n_{k+1}} \cdots p_{k+m-1}^* (C_{k+m}^*)^{n_{k+m}}$$ where $n_0+n_1+\cdots +n_{k+m}=n-(k+m-1)$. These are distinct elements of the  basis $\mathcal{B}$ and there are $n+1 \choose k+m$ of them for $n$ large (this counting problem is equivalent to counting the number of ways of placing $n-(k+m-1)$ identical coins into $k+m$ distinct pockets). Hence $dim (W^n)$ has a lower bound which is a polynomial of degree $k+m$ in n. Thus  $GKdim L_{\mathbb{F}}(\Gamma) \geq k+m$. If $ht(\Gamma)=2k-1$  then we can find distinct cycles $C_1 \leadsto C_2 \leadsto \cdots \leadsto C_k$. Setting  $C_{k+1}:=C_{k-1}$,  $C_{k+2}:=C_{k-2}$, $\cdots C_{2k-1} =C_1$, we get $GKdim L_{\mathbb{F}}(\Gamma) \geq 2k-1=ht(\Gamma)$.\\

 Similarly, given $C_1 \leadsto C_2 \leadsto \cdots \leadsto C_k \leadsto w$  and $C_{k+m} \leadsto C_{k+m-1} \leadsto \cdots \leadsto C_{k+1} \leadsto w $ be distinct cycles with $w$ a sink and $p_1$ a path containing no cycles from $sC_1$ to $sC_2$, also $p_2$ a path (containing no cycles) from $sC_2$ to $sC_3$, $\cdots , p_k$ from $sC_k$ to $w$, $p_{k+1}$ from $sC_{k+1}$ to $w$, $\cdots$, $p_{k+m}$ from $sC_{k+m}$ to $sC_{k+m-1}$.
The elements in $W^n$ of the form $$C_1^{n_1}p_1C_2^{n_2}p_2 \cdots C_k^{n_k}p_kw^{n_0}p_{k+1}^*(C_{k+1}^*)^{n_{k+1}} \cdots p_{k+m}^* (C_{k+m}^*)^{n_{k+m}}$$  $$=C_1^{n_1}p_1C_2^{n_2}p_2 \cdots C_k^{n_k}p_kp_{k+1}^*(C_{k+1}^*)^{n_{k+1}} \cdots p_{k+m}^* (C_{k+m}^*)^{n_{k+m}}$$ where $n_0+n_1+\cdots +n_{k+m}=n-(k+m)$.
 These are distinct elements of the  basis $\mathcal{B}$ and there are $n \choose k+m$ of them for $n$ large (this counting problem is equivalent to counting the number of ways of placing $n-(k+m)$ identical coins into $k+m+1$ distinct pockets). Hence $dim (W^n)$ has a lower bound which is a polynomial of degree $k+m$ in n. Thus  $GKdim L_{\mathbb{F}}(\Gamma) \geq k+m$. If $ht(\Gamma)=2k$  then we can find distinct cycles $C_1 \leadsto C_2 \leadsto \cdots \leadsto C_k \leadsto w$ where $w$ is a sink. Setting  $C_{k+1}:=C_k$,  $C_{k+2}:=C_{k-1}$, $\cdots C_{2k} =C_1$, we get $GKdim L_{\mathbb{F}}(\Gamma) \geq 2k=ht(\Gamma)$.\\\\                                               
To show that $ht(\Gamma) \geq GKdim L_{\mathbb{F}}(\Gamma) $, we note that an arbitrary element of $W^n$ for $n$ large is in the span of elements of the form: $$pq^*=p_0C_1^{n_1}p_1C_2^{n_2} \cdots p_{k-1}C_k^{n_k}p_{k}p_{k+1}^*(C_{k+1}^*)^{n_{k+1}} p_{k+2}^* \cdots (C_{k+m}^*)^{n_{k+m}}p_{k+m+1}^*$$ 
where $n_1+n_2+ \cdots + n_{k+m} \leq n$ and $p_0, p_1, \cdots p_{k+m+1}$ do not contain any cycles.\\

For the chains of the distinct cycles $C_1 \leadsto C_2 \leadsto \cdots \leadsto C_k$ and $ C_{k+m} \leadsto \cdots \leadsto C_{k+1}$ there are finitely many choices for each $p_i$ where $0\leq i \leq k+m+1$, the number of these choices depends on $k+m$, but not on $n$. Fixing $p_0, p_1, \cdots , p_{k+m+1}$, the number of possibilities for $n_1, n_2, \cdots n_{k+m}$ is bounded by a polynomial in $n$ of degree $k+m$ (counting the number of ways of placing at most $n$ identical coins into $k+m$ distinct pockets). Since $k+m \leq ht(\Gamma)$ we have a polynomial upper bound which is the sum of the polynomials of degree $\leq ht(\Gamma)$ corresponding to the choice of $p_0, p_1,\cdots, p_{k+m+1}$. (The bound on the number of these polynomials does not depend on $n$.) Therefore $GKdim L_{\mathbb{F}}(\Gamma) = ht(\Gamma)$.
\end{proof}\\

Some early results on Leavitt path algebras of finite digraphs are easy consequences of Theorem \ref{height}.  $ GKdim L_{\mathbb{F}}(\Gamma)=0=ht(\Gamma)$ if and only if $\Gamma$ is acyclic. Hence $L_{\mathbb{F}}(\Gamma)$ is finite dimensional if and only if $\Gamma$ is acyclic \cite{aam07}, since the Gelfand-Kirillov dimension of any algebra $A$ is 0 if and only if $ dim^{\mathbb{F}}(A)$ is finite. \\

$GKdim L_{\mathbb{F}}(\Gamma)=1 $ if and only if $ \Gamma$ has at least one cycle but cycles have no exits, that is, $ht(\Gamma)=1$. In this case $L_{\mathbb{F}}(\Gamma)$ is a direct sum of matrix algebras over $\mathbb{F} [x^{-1},x] \> \>$ and possibly $\mathbb{F}$ \cite{aam08}. This is also a consequence of Theorem \ref{basis} above. Hence representations of  $L_{\mathbb{F}}(\Gamma)$ with $GKdim L_{\mathbb{F}}(\Gamma)\leq 1$ are well understood.
Little is known beyond the classification of simple $L_{\mathbb{F}}(\Gamma)$-modules when $ GKdim L_{\mathbb{F}}(\Gamma)\geq 2$. These will be constructed explicitly in the next section. 

\begin{examples} \label{ornek}
The graph $C^*$-algebras of the following digraphs (which are the completions of Leavitt path algebras with complex coefficients) are quantum disks, quantum spheres and quantum real projective spaces \cite{hs02}. The graph $C^*$-algebra of $qD^2$, the quantum 2-disk is also the Toeplitz algebra, the Leavitt path algebra of this digraph is isomorphic to the Jacobson algebra $\mathbb{F} \langle x,y \rangle / (1-xy)$. \\\\

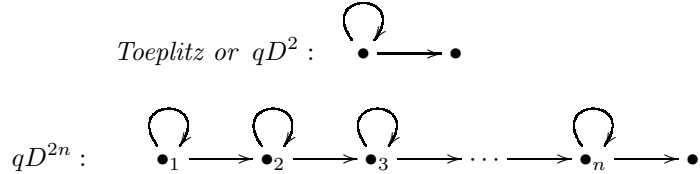
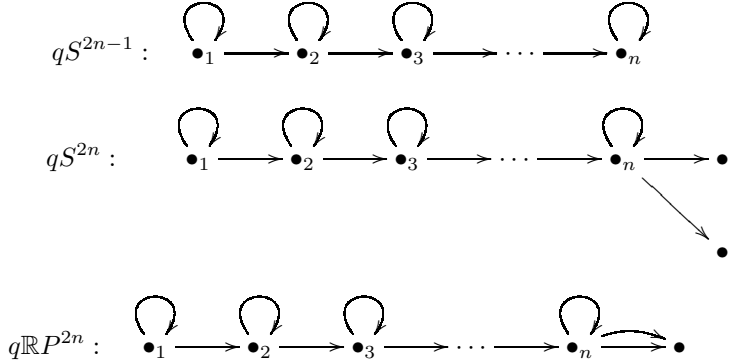
$$\textit{Toeplitz  or }\>  qD^2 : \quad \xymatrix{ {\bullet} \ar@(ul,ur)
 \ar@{->}[r]  & \bullet  } $$
 
  $$\qquad \quad \qquad qD^{2n} : \qquad \xymatrix{ {\bullet}_1 \ar@(ul,ur)
 \ar@{->}[r]   &{\bullet}_2 \ar@(ul,ur)
 \ar@{->}[r]    & {\bullet}_3 \ar@(ul,ur)\ar@{->}[r] &\cdots \ar@{->}[r] & \bullet_n \ar@(ul,ur) \ar@{->}[r] 
 %\ar@{->}[dr] 
 &\bullet }$$
 
$$\qquad qS^{2n-1} : \quad \xymatrix{ {\bullet}_1 \ar@(ul,ur) \ar@{->}[r]   &{\bullet}_2 \ar@(ul,ur)  \ar@{->}[r]    & {\bullet}_3 \ar@(ul,ur)\ar@{->}[r] &\cdots \ar@{->}[r] & \bullet_n \ar@(ul,ur) }$$
 
 $$\qquad \quad \qquad qS^{2n} : \qquad \xymatrix{ {\bullet}_1 \ar@(ul,ur)
 \ar@{->}[r]   &{\bullet}_2 \ar@(ul,ur)
 \ar@{->}[r]    & {\bullet}_3 \ar@(ul,ur)\ar@{->}[r] &\cdots \ar@{->}[r] & \bullet_n \ar@(ul,ur) \ar@{->}[r] \ar@{->}[dr] &\bullet \\
               &&& & & \bullet }$$
               
                  $$\qquad q\mathbb{R}P^{2n} : \quad \xymatrix{ {\bullet}_1 \ar@(ul,ur)
 \ar@{->}[r]   &{\bullet}_2 \ar@(ul,ur)
 \ar@{->}[r]    & {\bullet}_3 \ar@(ul,ur)\ar@{->}[r] &\cdots \ar@{->}[r] & \bullet_n \ar@(ul,ur) \ar@{->}[r]  \ar@/^/[r]  &\bullet }$$
 
\noindent
 Note that $ GKdim L_{\mathbb{F}}(\Gamma)=ht(\Gamma)$ is the dimension of the quantum space for these digraphs.
             \end{examples}

\begin{theorem} \label{k-height}
If $\Gamma$ is a finite digraph, $\mathbb{F}$ is a field and ${\bf k}$ is a commutative $\mathbb{F}$-algebra with 1 then
$$ GKdim L_{{\bf k}}(\Gamma)= GKdim L_{\mathbb{F}}(\Gamma) + GKdim ({\bf k}). $$
\end{theorem}

\begin{proof}
If $GKdim L_{\mathbb{F}}(\Gamma)$ or $GKdim ({\bf k})$ is infinite then $GKdim L_{\bf k}(\Gamma)$ is also infinite since the former are subalgebras of the latter. When both $GKdim L_{\mathbb{F}}(\Gamma)$ and $GKdim ({\bf k})$ are  finite, let $U$ be a finite dimensional generating $\mathbb{F}$-subspace of ${\bf k}$ containing 1 and let $W$ be the span of all the cycles of $\Gamma$, their duals, all paths not containing any cycles and their duals as in the proof of Theorem \ref{height}. We will use the generating subspace $U\otimes W$ of ${\bf k} \otimes_{\mathbb{F}} L_{\mathbb{F}}(\Gamma) \cong L_{{\bf k}}(\Gamma)$ which contains $1\otimes 1=1 $, to compute $GKdim L_{{\bf k}}(\Gamma)$. Now $dim^{\mathbb{F}} (U\otimes W)^n= (dim^{\mathbb{F}} U^n)( dim^{\mathbb{F}} W^n)$. Also $\limsup_{n \to \infty} \frac{log (dim(U^n))}{ log (n)}= GKdim ({\bf k})$ and, as shown in the proof of Theorem \ref{height}, $\lim_{n \to \infty} \frac{log (dim(W^n))}{ log (n)}=ht(\Gamma)$ yielding that $ GKdim L_{{\bf k}}(\Gamma)= ht(\Gamma) + GKdim ({\bf k}) =GKdim L_{\mathbb{F}}(\Gamma)+ GKdim ({\bf k})$.
\end{proof}\\

\begin{remark} If $\mathbb{F}$ is a field, ${\bf k}$ is a commutative $\mathbb{F}$-algebra with 1 and $\Gamma$ is a finite digraph whose cycles are pairwise disjoint then
 $GKdim \> {\bf k}\Gamma \>= \lceil ht(\Gamma)/2 \rceil+ GKdim ({\bf k})$. The proof is similar to those of Theorem \ref{height} and Theorem \ref{k-height} but considerably easier since there are no dual paths and the set of paths is already a basis for the path algebra.
\end{remark}

 \newpage 
 \medskip
\section{Simple $L(\Gamma)$-modules and Extensions}
\medskip

\subsection{Simple $L(\Gamma)$-modules }

We want to give an explicit description of all simple $L_{{\bf k}}(\Gamma)$-modules up to isomorphism when $\Gamma$ is a finite digraph whose cycles are pairwise disjoint.
The following Lemma is stated and proven in greater generality than we need, since it may be of independent interest.\\

An \textit{exclusive cycle} is a cycle which does not intersect any other cycle in $\Gamma$.\\

\begin{lemma} \label{Preclass} Let $C$ be an exclusive  cycle in $\Gamma$,  $\> f(x)=1-g(x)$ with $0\neq g(x) \in x{\bf k}[x]$ and  $M=vL_{{\bf k}}(\Gamma)/f(C^*)L_{{\bf k}}(\Gamma)$ where $v:=sC$ and $f(C^*):=v-g(C^*)$. Then\\

(i) $V_M = V_{\leadsto C}$\\

(ii) $vL_{{\bf k}}(\Gamma) / f(C^*)L_{{\bf k}}(\Gamma) \cong vL_{{\bf k}}(\Gamma_{\leadsto C})/f(C^*)L_{{\bf k}}(\Gamma_{\leadsto C})$
 as $L_{{\bf k}}(\Gamma)$-modules.

\end{lemma} 
\begin{proof}Let $L\Gamma$ denote $L_{{\bf k}}(\Gamma)$. (i) For all $u \in V$ the sequence 
$vL\Gamma u \stackrel{f(C^*)\cdot }{\longrightarrow}vL\Gamma u \longrightarrow Mu \longrightarrow 0$ is exact. $vL\Gamma u $ is spanned by $\{pq^*\vert \> sp=v, \> \>tp=tq,\> \>  sq=u\}$. If $u\notin V_{\leadsto C}$ and $u \stackrel{q}{\leadsto} tq=tp$ then $tp \notin V_{\leadsto C}$. Hence $(C^n)^*p=(C^*)^np=0$ for $n=l(p)+1$ because $C^n$ is too long to be an initial segment of $p$ and $p$ can not be an initial segment of $C^n$ since $tp \notin V_{\leadsto C}$. Now $f(C^*)[v +\Sigma_{k=1}^{n-1}g(C^*)^k]pq^*=(v-g(C^*)^n)pq^*=vpq^*=pq^*$.
Hence $vL\Gamma u \stackrel{f(C^*)\cdot }{\longrightarrow}vL\Gamma u$ is onto and thus $Mu=0$. 
That is $V_M\subseteq V_{\leadsto C}$. If $u\in V_{\leadsto C}$ then there is a path $q$  such that $sq=u$ and $tq=v$. Now, 
$Mv \stackrel{\cdot q^*}{\longrightarrow} Mu$ is one-to-one, hence $Mu\neq 0$, that is, $V_{\leadsto C} = V_M$.\\

(ii) Let $I$ be the ideal generated by 
$V\setminus V_{\leadsto C}$. Since $I$ annihilates $M$, we get that $M=vL\Gamma /(vI+f(C^*)L\Gamma)$, that is, $vI\subseteq f(C^*)L\Gamma$. Consider the composition $vL\Gamma_{\leadsto C} \stackrel{\backsimeq\quad }{\longrightarrow} vL\Gamma / vI \longrightarrow M$ where the first isomorphism is the restriction of the isomorphism $L\Gamma_{\leadsto C} \longrightarrow L\Gamma /I$ restricted to $vL\Gamma_{\leadsto C}$. This composition is an epimorphism because both homomorphisms are onto. If $\alpha \in vL\Gamma_{\leadsto C}$ is in the kernel of this composition then its image is $f(C^*)\beta $ for some $\beta$ in $L\Gamma$. We may assume that $\beta $ is in $vL\Gamma$, replacing $\beta$ with $v\beta$ if necessary, since $f(C^*)=f(C^*)v$. Now $\beta =\gamma +\delta$ where $\gamma$ involves only paths in $\Gamma_{\leadsto C}$ and $\delta$ is in $ vI$. Let $\epsilon$ be the element in $vL\Gamma_{\leadsto C}$ which has the same expression as $\gamma$ in $vL\Gamma$. The image of $f(C^*)\epsilon$ is the same as the image of $\alpha$ in $vL\Gamma /vI$. Hence $\alpha = f(C^*)\epsilon$, so $f(C^*) \in vL\Gamma_{\leadsto C}$ generates the kernel of the epimorphism  $vL\Gamma_{\leadsto C} \longrightarrow M$. Thus $vL\Gamma_{\leadsto C} / f(C^*)L\Gamma_{\leadsto C} \cong M $.
\end{proof}

\begin{theorem} \label{Classification 1}
Let $\Gamma$ be a finite digraph whose cycles are pairwise disjoint and let $M$ be a simple $L_{\bf k}(\Gamma)$-module. \\
(i) If there is a sink $w$ in $\Gamma$ with $Mw\neq 0$ then $V_M=V_{\leadsto w}$ and   $$M\cong  \frac{wL_{\bf k}(\Gamma)}{\mathfrak{m}wL_{\bf k}(\Gamma)}\cong  wL_{^{{\bf k}}\!/\!_{\mathfrak{m}}}(\Gamma)$$
for a unique maximal ideal $\mathfrak{m} $ of ${\bf k}$. \\

Conversely, $wL_{^{{\bf k}}\!/\!_{\mathfrak{m}}}(\Gamma)$ is a simple $L_{\bf k}(\Gamma)$-module for each maximal ideal $\mathfrak{m}$ of ${\bf k}$ and each sink $w$ in $\Gamma$. This gives a one-to-one correspondence between isomorphism classes of simple $L_{\bf k}(\Gamma)$-modules containing a sink in their support and $$\{ \mathfrak{m}\vartriangleleft {\bf k} \>  \vert \> \mathfrak{m} \textit{ maximal } \} \times \{ w \in V \> \vert \> w \textit{ a sink }\}$$ 
i.e., the  cartesian product of the maximal spectrum of ${\bf k}$ with the set of sinks in $\Gamma$.\\

\noindent
(ii) If $Mw=0$ for every sink $w$ in $\Gamma$ then there is a unique cycle $C$ in $\Gamma_M =\Gamma_{\leadsto C}$ with $v:=sC$ and 

$$M\cong \frac{{\bf k}[x,x^{-1}]}{\mathfrak{m}} \bigotimes_{{\bf k}[x,x^{-1}]} vL_{\bf k}(\Gamma_{\leadsto C})$$ 
 for a unique maximal ideal $\mathfrak{m} $ of ${\bf k}[x,x^{-1}]$. Here $vL_{\bf k}(\Gamma_{\leadsto C})$ is a ${\bf k}[x,x^{-1}]$-module via  ${\bf k}[x,x^{-1}] \cong vL_{\bf k}(\Gamma_{\leadsto C})v$ by Proposition \ref{Path}(ii) and  $vL_{\bf k}(\Gamma_{\leadsto C})$ is a left $L_{\bf k}(\Gamma)$-module since it is isomorphic to $vL_{\bf k}(\Gamma)/ vI$ where $I=(V\setminus V_{\leadsto C})$ by Proposition \ref{hereditary}. \\
 
 Conversely, $({\bf k}[x,x^{-1}] \diagup \mathfrak{m}) \bigotimes_{{\bf k}[x,x^{-1}]} vL_{\bf k}(\Gamma_{\leadsto C})$ is a simple $L_{\bf k}(\Gamma)$-module for each maximal ideal $\mathfrak{m}$ of ${\bf k}[x,x^{-1}]$ and each cycle $C$ in $\Gamma$. This gives a one-to-one correspondence between isomorphism classes of simple $L_{\bf k}(\Gamma)$-modules not containing a sink in their support and $$\{ \mathfrak{m}\vartriangleleft {\bf k}[x,x^{-1}]  \> \vert \>  \mathfrak{m} \textit{ maximal } \} \times \{ C  \> \vert \> C \textit{ a cycle in } \Gamma \}$$ 
i.e., the  cartesian product of the maximal spectrum of ${\bf k}[x,x^{-1}]$ with the set of cycles in $\Gamma$.

\end{theorem}

\begin{proof}
We pick a vertex $v$ which is minimal with respect to $\leadsto$ such that $Mv \neq 0$ and $0\neq m \in Mv$. Since $M$ is simple, $M=mL$. So $M$ is spanned by $mpq^*$ where $sp=v$ and $tp=tq=sC$ for some cycle $C$ with $sC\in V_M$ or $tp=tq$ is a sink in $V_M$ by Theorem \ref{basis}. We may assume that $v$ is a sink (case (i)) or $v=sC$ (case (ii)), replacing $v$ if necessary because $v$ is minimal in $V_M$.
In both cases $V_M=V_{\leadsto v}$ because $mpq^* \in Msq$ and $sq \leadsto v$ so $V_M \subseteq V_{\leadsto v}$ and $Msq\neq 0$ by Proposition \ref{pp^*}(i) . \\

If $v\in V_M$ is a sink then via the one-to-one correspondence given in Theorem \ref{Ser2} between simple $L_{\bf k}(\Gamma)$-modules generated by $Mv$ and simple $vL_{\bf k}(\Gamma)v\cong {\bf k}$-modules we get that 
$$M\cong \frac{{\bf k}}{\mathfrak{m}} \bigotimes_{\bf k} vL_{\bf k} (\Gamma) \cong \frac{vL_{\bf k}(\Gamma)}{\mathfrak{m}vL_{\bf k}(\Gamma)}\cong vL_{^{{\bf k}}\!/\!_{\mathfrak{m}}}(\Gamma)$$
using ${\bf k}\cong vL_{\bf k}(\Gamma)v$ by Proposition \ref{pp^*}(iii), that simple ${\bf k}$-modules are of the form ${\bf k} \diagup \mathfrak{m}$ and the restriction of the isomorphism between $L_{\bf k}(\Gamma) \diagup \mathfrak{m}L_{\bf k}(\Gamma)$ and $ L_{^{{\bf k}}\!/\!_{\mathfrak{m}}}(\Gamma)$ to $vL_{^{{\bf k}}\!/\!_{\mathfrak{m}}}(\Gamma)$. \\

If $v=sC$ then $vL_{\bf k}(\Gamma)v \cong {\bf k}[x,x^{-1}]$ by Proposition \ref{Path}(ii) and we get 
  $$M\cong \frac{{\bf k}[x,x^{-1}]}{\mathfrak{m}} \bigotimes_{{\bf k}[x,x^{-1}]} vL_{\bf k}(\Gamma_{\leadsto C})$$
  where $\mathfrak{m}$ is the unique maximal ideal of ${\bf k}[x,x^{-1}]$ such that the simple ${\bf k}[x,x^{-1}]$-module  ${\bf k}[x,x^{-1}] \diagup{\mathfrak{m}}$ corresponds to $M$ via the functor $\underline{\>\>\>} \otimes vL_{\bf k}(\Gamma)$ of Theorem \ref{Ser2}.\\
  
  One-to-one correspondences between isomorphism classes of simple $L_{\bf k}(\Gamma)$-modules of $M$ and the relevant cartesian product given above, are gotten by observing that the support $V_M$ determines the relevant sink or cycle uniquely and that the isomorphism class of a simple $R$-module for a commutative ring $R$ with 1 determines a unique maximal ideal of $R$, namely its annihilator.  
 \end{proof}

\begin{remark}
The classification Theorem above yields explicit bases for simple $L_{\bf k}(\Gamma)$-modules:\\ 
\noindent
 (i) If $V_M$ contains a sink $w$ then $$M\cong wL_{^{{\bf k}}\!/\!_{\mathfrak{m}}}(\Gamma) $$
 has $wP_w^*$ as a ${\bf k} \diagup \mathfrak{m}$-basis by Proposition \ref{Path}(iii).\\
 \noindent
 (ii) If $V_M =V_{\leadsto C}$ then 
 
 $$ M = \frac{{\bf k}[x,x^{-1}]}{\mathfrak{m}} \bigotimes_{{\bf k}[x,x^{-1}]} vL_{\bf k}(\Gamma_{\leadsto C}) $$ has basis 
   $vP_C^*$ over ${\bf k}[x,x^{-1}]\diagup \mathfrak{m}$ by Corollary   \ref{corner}, identifying $v:=sC$ with $(1+\mathfrak{m})\otimes v$.\\
 \end{remark}

When  ${\bf k}$ is a field then the Classification Theorem above simplifies and it is more explicit: 

\begin{corollary} \label{Classification 2}

Let $\Gamma$ be a finite digraph whose cycles are pairwise disjoint, $\mathbb{F}$ a field and $M$ a simple $L_{\mathbb{F}}(\Gamma)$-module. \\

\noindent 
(i) If there is a sink $w$ in $\Gamma$ with $Mw\neq 0$ then $V_M=V_{\leadsto w}$ and   $$M\cong  wL_{\mathbb{F}}(\Gamma)\> .$$ 

Conversely, the projective $L_{\mathbb{F}}(\Gamma)$-module $wL_{\mathbb{F}}(\Gamma)$ is simple for each sink $w$ in $\Gamma$. This gives a one-to-one correspondence between  isomorphism classes of simple projective $L_{\mathbb{F}} (\Gamma)$-modules and the set of sinks in $\Gamma$.\\

\noindent
(ii) If $Mw=0$ for every sink $w$ in $\Gamma$ then there is a unique cycle $C$ in $\Gamma_M =\Gamma_{\leadsto C}$ and
$$M\cong \frac{vL_{\mathbb{F}}(\Gamma)}{f(C^*)L_{\mathbb{F}}(\Gamma)}$$ where $v:=sC$, $\> f(x)=1-g(x)$ with $0\neq g(x) \in x\mathbb{F}[x]$ and $f(C^*):=v-g(C^*)$. In fact, $f(x)= det (v-xC)$, the characteristic polynomial of the $\mathbb{F}$-linear operator on $Mv$ defined by right multiplication with $C$. \\

  Conversely, $vL_{\mathbb{F}}(\Gamma)/f(C^*)L_{\mathbb{F}}(\Gamma)$ is a simple $L_{\mathbb{F}}(\Gamma)$-module for each cycle $C$ in $\Gamma$ with $v=sC$ and each irreducible polynomial $f(x)$ with $f(0)=1$. This gives a one-to-one correspondence between isomorphism classes of nonprojective simple $L_{\mathbb{F}}(\Gamma)$-modules  and $$\{ f(x) \in  \mathbb{F}[x]  \> \vert \> f(x) \textit{ irreducible and } f(0)=1 \} \times \{ C \> \vert \> C \textit{ a cycle in } \Gamma \}.$$
\end{corollary}

\begin{proof}
(i) The only maximal ideal of $\mathbb{F}$ is $0$ so $M\cong wL_{\mathbb{F}}(\Gamma)$ by Theorem \ref{Classification 1}. Also $V_M=V_{\leadsto w}$ by Theorem \ref{Classification 1}.\\

By Proposition \ref{Path}(iii)  $wL_{\mathbb{F}}(\Gamma)$ is simple for each sink $w$. The one-to-one correspondence is given by $w \mapsto [wL_{\mathbb{F}}(\Gamma)]$, the isomorphism class of $wL_{\mathbb{F}}(\Gamma)$, and $[P]$, the isomorphism class of a simple projective, corresponds to the unique sink $w$ in $V_P$, as in Theorem \ref{Classification 1}.\\

(ii) By Proposition \ref{Path}(ii) we have $C^* $ corresponding to $x$ in the isomorphism between $vL_{\mathbb{F}}(\Gamma)v$ and $\mathbb{F}[x,x^{-1}]$. Maximal ideals $\mathfrak{m}$ of $\mathbb{F}[x,x^{-1}]$ are uniquely determined by irreducible polynomials $f(x) \in \mathbb{F}[x]$ where $f(x)=1-g(x)$ and $0\neq g(x) \in x\mathbb{F}[x]$, so $\mathfrak{m}= (f(x))$. Hence 
$$M\cong \frac{\mathbb{F}[x,x^{-1}]}{(f(x))} \bigotimes_{\mathbb{F}[x,x^{-1}]} vL_{\mathbb{F}}(\Gamma_{\leadsto C}) \cong \frac{vL_{\mathbb{F}}(\Gamma_{\leadsto C})}{f(C^*)L_{\mathbb{F}}(\Gamma_{\leadsto C})} \cong \frac{vL_{\mathbb{F}}(\Gamma)}{f(C^*)L_{\mathbb{F}}(\Gamma)}$$
where the first isomorphism is given by Theorem \ref{Classification 1}(ii) and the isomorphism by Lemma \ref{Preclass}(ii).\\

$vL_{\mathbb{F}}(\Gamma) \diagup f(C^*)L_{\mathbb{F}}(\Gamma)$ is simple for each cycle $C$ in $\Gamma$ by Theorem \ref{Classification 1}(ii). Since $Mv\cong \mathbb{F}[x,x^{-1}] \diagup (f(x))$ an $\mathbb{F}$-basis for $Mv$ is given by $\{(C^*)^k\}$ for $k=1,2, \cdots, deg(f)$. Right multiplication by $C$ sends $C^*$ to $g(C^*)$ and $(C^*)^k$ to $(C^*)^{k-1}$ for $k=2, \cdots, deg(f)$. A standard computation yields that $det(v-xC)=f(x)$. \\

Given a simple $L_{\mathbb{F}}(\Gamma)$-module $M$ we recover $C$ as the unique minimal cycle in $\Gamma_M$ and $f(x)$ as the characteristic polynomial of right multiplication by $C$ on $Mv$, establishing the one-to-one correspondence. Note that $f(x)$ is independent of the choice of the base vertex $v=sC$ 
since $p^*Cp$ on $Mtp$ is "conjugate" to $C$ on $Mv$ where $p$ is the unique segment of $C$ connecting $v$ to $tp$.
\end{proof}\\

Corollary 4.6 in \cite{ar14} states that all simple  $L_{\mathbb{F}}(\Gamma)$-modules are Chen modules when $\mathbb{F}$ is a field and $\Gamma$ is a finite digraph whose cycles are pairwise disjoint. Simple projective $L_{\mathbb{F}}(\Gamma)$-modules are isomorphic to $wL_{\mathbb{F}}(\Gamma)$-modules for some sink $w$ by Proposition \ref{Path}(iii). For nonsimple projective $L_{\mathbb{F}}(\Gamma)$-modules the correspondence between their statement and our classification is as follows.\\

Chen modules are formal linear combinations of infinite paths in a single tail equivalence class \cite{che15}.
They are usually defined as left modules with $(pq^*)\alpha$ defined using a sufficiently long initial segment of the infinite path $\alpha$ for $pq^*\in L_{\mathbb{F}}(\Gamma)$. A left $L_{\mathbb{F}}(\Gamma)$-module can be made a right module $M$ by defining $m\cdot a:= a^*m$ for $a \in L_{\mathbb{F}}(\Gamma)$ and $m \in M$. When the cycles in $\Gamma$ are pairwise disjoint all infinite paths are eventually periodic, being tail equivalent to $C^{\infty}$ for some cycle $C$. These are called rational Chen modules and they are isomorphic to $vL_{\mathbb{F}}(\Gamma) \diagup (v-C^*)L_{\mathbb{F}}(\Gamma)$ where $v=sC$ in Corollary \ref{Classification 2}, i.e., $f(x)=1-x$. When $f(x)=1-\lambda x$ with $0\neq \lambda \in \mathbb{F}$ we get the twisted rational Chen modules. When $\mathbb{F}$ is algebraically closed these are all possible irreducible polynomials $f(x)$. 
%When $\mathbb{F}$ is not algebraically closed the modules 
Strictly speaking, the modules $vL_{\mathbb{F}}(\Gamma) \diagup f(C^*)L_{\mathbb{F}}(\Gamma)$ with $deg f(x) >1$ are not Chen modules, they can be obtained from twisted rational Chen modules by extending the coefficient field $\mathbb{F}$. \\

\subsection{Extensions}

\begin{lemma} \label{re-ext}  If $A$ is an $L_{{\bf k}}(\Gamma)$-module and $B:= vL_{{\bf k}}(\Gamma) \diagup f(C^*)L_{{\bf k}}(\Gamma)$ where $C$ is an exclusive cycle in $\Gamma$, $v=sC$ and $f(C^*)=v-g(C^*)$ with $0\neq g(x) \in x{\bf k}[x]$ then $Ext(B,A)  \cong Av/Af(C^*)$.
In particular, $Ext(B,A)=0$ if $Av=0$.
\end{lemma}

\begin{proof}
We use the exact sequence $vL_{{\bf k}}(\Gamma) \stackrel{f(C^*) \cdot}{\longrightarrow} vL_{{\bf k}}(\Gamma) \longrightarrow B \longrightarrow 0$ to compute $Ext(B,A)$. Hence $Ext(B,A) \cong Coker\{ Hom^{L}(vL_{{\bf k}}(\Gamma), A) \longrightarrow Hom^{L}(vL_{{\bf k}}(\Gamma), A)\}$. Since  $Hom^{L}(vL_{{\bf k}}(\Gamma), A) \cong Av$ and $Hom^{L}(f(C^*) \cdot , A)= \cdot f(C^*)$ on $Av$, we are done.
\end{proof}

\begin{remark}
In fact, $0 \longrightarrow vL_{{\bf k}}(\Gamma) \stackrel{f(C^*) \cdot}{\longrightarrow} vL_{{\bf k}}(\Gamma) \longrightarrow B\longrightarrow 0$ is a projective resolution of $B$: as in the proof of Lemma \ref{Preclass} $vL_{{\bf k}}(\Gamma)u \stackrel{f(C^*) \cdot}{\longrightarrow} vL_{{\bf k}}(\Gamma)u$ is an isomorphism  if $u \notin V_{\leadsto v}$ and we may assume that $\Gamma =\Gamma_{\leadsto C}$. On $ vL_{{\bf k}}(\Gamma)v \cong {\bf k}[x,x^{-1}]$ left multiplication by $f(C^*)$ is identified with $f(x) \cdot$, hence it is one-to-one. If $u \in V_{\leadsto C}$ then $f(C^*) \cdot$ on $vL_{{\bf k}}(\Gamma)u= \bigoplus_{q \in P_C} vL_{{\bf k}}(\Gamma)q^*$ is a direct sum of ${\bf k}$-monomorphisms by Lemma \ref{Temel}, thus it is one-to-one.  \\
\end{remark}

\noindent
Let $Q_D^{sC}:=\{ q \in Path \Gamma \> \vert \> sq=sC , tq=sD , \> q\neq q_1D, \> q\neq Cq_2 \text{ for any } q_1, q_2 \in Path\Gamma\}$ and 
$Q_w^{sC}:=\{ q \in Path \Gamma \> \vert \> sq=sC \>, tq=w\> ,  \> q\neq Cq' \text{ for any } q' \in Path\Gamma\}$.

\begin{lemma} \label{re-ext2}
Let $\Gamma$ be a finite digraph, $A$ be an 
$L\Gamma:=L_{{\bf k}}(\Gamma)$-module generated by $Aw$  where $w$ is either (i) a sink or (ii) the source of a cycle with no exit. If $C$ is an exclusive cycle with $v=sC$, $f(C^*)=v-g(C^*)$ where $f(x)= 1-g(x)$ and $0\neq g(x) \in x{\bf k}[x]$ then $Ext (vL\Gamma \diagup f(C^*) L\Gamma, A )$ is isomorphic to a direct sum of copies of $Aw$ indexed by $\{C^nq \> \vert \> q \in Q_w^v, \> 0\leq n < deg(f) \}$.
\end{lemma}

\begin{proof} Let $B:=vL\Gamma\diagup f(C^*) L\Gamma$. 
By Lemma \ref{Temel}
$$Av=\bigoplus_{p \in P_w^v} Awp^*=\bigoplus_{q \in Q_w^v, \> n \in \mathbb{N}}Aw(C^nq)^*$$ 
where $C^0:=v$ as usual. Hence  
$$ Ext (B,A) \> \> =\bigoplus_{q \in Q_w^v, \> 0\leq n < deg(f)  }Aw(C^nq)^* $$
by Lemma \ref{re-ext}, because $\{ C^n \> \vert \> n \in \mathbb{N} \}$  is ${\bf k}$-linearly independent and $$\frac{{\bf k}[x]}{(f(x))}\> \cong \bigoplus_{n=0}^{ deg(f)-1 } {\bf k}x^n \textit{ as a }{\bf k}-\text{module.}$$ 
Finally, $Aw(C^nq)^* \cong Aw$ by Proposition \ref{pp^*}(i). Note that if $v \notin V_{\leadsto w}$ then $Q_w^v= \emptyset$ and $Ext ( B, A ) =0$ by Lemma \ref{re-ext} since $Av=0$, so the statement holds.
\end{proof}

For the rest of this section ${\bf k}=\mathbb{F}$ is a field and $\Gamma$ is a finite digraph whose cycles are pairwise disjoint. We want to compute $Ext(B,A)$ when $A$ and $B$ are of the form $wL_{\mathbb{F}}(\Gamma)$ with $w$ a sink or $vL_{\mathbb{F}}(\Gamma)\diagup f(C^*)L_{\mathbb{F}}(\Gamma)$ where $v=sC$, $f(C^*)=v-g(C^*)$ with $f(x)= 1-g(x)$ and $0\neq g(x) \in x{\mathbb{F}}[x]$. First we take care of a few easy cases: \\

(i) If $B=wL_{\mathbb{F}}(\Gamma)$ then $Ext(B,A)=0$ since $B$ is projective. \\

(ii) If $v \notin V_A $ where $B=vL_{\mathbb{F}}(\Gamma)\diagup f(C^*)L_{\mathbb{F}}(\Gamma)$ then $Ext(B,A)=0$ by Lemma \ref{re-ext}.\\

(iii) Let $B=v_1L_{\mathbb{F}}(\Gamma)\diagup f_1(C_1^*)L_{\mathbb{F}}(\Gamma)$ and $A$ be either $wL_{\mathbb{F}}(\Gamma)$ with $w$ a sink or $A=v_2L_{\mathbb{F}}(\Gamma)\diagup f_2(C_2^*)L_{\mathbb{F}}(\Gamma)$ with $v_i=sC_i$ and $f_i$ as above for $i=1,2$. If there is a cycle strictly between $C_1$ and $w$,  or $C_1$ and $C_2$ then $Ext(B,A)$ is infinite dimensional by Lemma \ref{re-ext2}.\\

 (iv) If $B=vL_{\mathbb{F}}(\Gamma)\diagup f_1(C^*)L_{\mathbb{F}}(\Gamma)$ and $A=vL_{\mathbb{F}}(\Gamma)\diagup f_2(C^*)L_{\mathbb{F}}(\Gamma)$ where $v=sC$ and $f_i(x)$ as above for $i=1,2$ then $dim^{\mathbb{F}}Ext(B,A)=deg(f_2)-deg (gcd(f_1,f_2))$. Here $Ext(B,A) \cong Ext \big( \frac{\mathbb{F}[x,x^{-1}] }{(f_1(x))}, \frac{\mathbb{F}[x,x^{-1}] }{(f_2(x))} \big) $ by Theorem \ref{Ser2}.\\

In the remaining cases $B=v_1L_{\mathbb{F}}(\Gamma)\diagup f_1(C_1^*) L_{\mathbb{F}}(\Gamma)$ and $v_2$ either a sink and $A= v_2L_{\mathbb{F}}(\Gamma)$; or the source of a cycle $C_2$ and $A=v_2L_{\mathbb{F}}(\Gamma) \diagup f_2(C_2^*)L_{\mathbb{F}}(\Gamma)$ where $v_i=sC_i$ with   $f_i(x)$ as above for $i=1, 2$ and $v_1$ \textit{covers} $v_2$, that is, $v_1 \leadsto v_2$ and there is no cycle between $v_1$ and $v_2$.

\begin{theorem} \label {katlilik} Let $\mathbb{F}$ be a field, $\Gamma$ be a finite digraph whose cycles are pairwise disjoint. If $A$ and $B$ are as above then 
 \[ dim^{\mathbb{F}} Ext (B,A)=\begin{cases} 
			deg(f_1) \> \vert Q_{v_2}^{v_1} \vert &  \qquad v_2 \text{ a sink }  \\
			deg(f_1) \> \vert Q_{v_2}^{v_1} \vert \> deg(f_2) &\qquad  v_2=sC_2 
	\end{cases}
		\]
\end{theorem}

\begin{proof} $A\cong v_2L_{\mathbb{F}}(\Gamma_{\leadsto v_2})$ (if $v_2$ is a sink) or $A \cong v_2L_{\mathbb{F}}(\Gamma_{\leadsto v_2}) \diagup f_2(C_2^*)L_{\mathbb{F}}(\Gamma_{\leadsto v_2})$ and $B\cong v_1L_{\mathbb{F}}(\Gamma_{\leadsto v_1})\diagup f_1(C_1^*) L_{\mathbb{F}}(\Gamma_{\leadsto v_1})$ by Lemma \ref{Preclass}. Since $v_1 \leadsto v_2$ we see that $\Gamma_{\leadsto v_1}$ is a subgraph of $\Gamma_{\leadsto v_2}$. Applying Lemma \ref{Preclass} again with $\Gamma_{\leadsto v_2}$ replacing $\Gamma$, we get that $B \cong v_2L_{\mathbb{F}}(\Gamma_{\leadsto v_2}) \diagup f_1(C_1^*)L_{\mathbb{F}}(\Gamma_{\leadsto v_2})$. Hence we may assume that $\Gamma =\Gamma_{\leadsto v_2}$. Using $dim^{\mathbb{F}} Av_2=1$ if $v_2$ is a sink, $dim^{\mathbb{F}} Av_2 =deg(f_2)$ if $v_2=sC_2$ the theorem follows from Lemma  \ref{re-ext2}.
\end{proof}\\

\section{Morita Invariants of $L_{\mathbb{F}}(\Gamma)$}

In this section  $\Gamma$ is a finite digraph whose cycles are  pairwise disjoint and ${\bf k}=\mathbb{F}$, a field.

\subsection{Invariant Filtrations of $\Gamma$, $L_{\mathbb{F}}(\Gamma)$ and $Mod_{L_{\mathbb{F}}(\Gamma)}$ }

 The sinks in $\Gamma$ are in one-to-one correspondence with isomorphism classes of simple projective $L_{\mathbb{F}}(\Gamma)$-modules. The cycles in $\Gamma$ can also be realized as isomorphism classes of certain finitely generated indecomposable modules defined purely in terms of $Mod_{L_{\mathbb{F}}(\Gamma)}$, the category of $L_{\mathbb{F}}(\Gamma)$-modules. This is achieved via a filtration of $Mod_{L_{\mathbb{F}}(\Gamma)}$ by Serre subcategories. In fact we will show that the poset of the sinks and the cycles of $\Gamma$ under $\leadsto$ is a Morita invariant of $L_{\mathbb{F}}(\Gamma)$.\\

\begin{theorem} \label{fgip}
If $P$ is a \textbf{f}initely \textbf{g}enerated \textbf{i}ndecomposable \textbf{p}rojective (\textbf{fgip}) module of $L_{\mathbb{F}}(\Gamma)$ then either (i) $P$ is simple and $P\cong wL_{\mathbb{F}}(\Gamma)$ for a unique sink $w$; or (ii) $P$ is not simple and $P\cong vL_{\mathbb{F}}(\Gamma)$  where $v\in V$ is on a unique cycle with no exit. This gives  one-to-one correspondences (i) between the isomorphism classes of simple projective $L_{\mathbb{F}}(\Gamma)$-modules and the sinks of $\Gamma$; (ii) between the isomorphism classes of non-simple fgip $L_{\mathbb{F}}(\Gamma)$-modules and the cycles of $\Gamma$ with no exits.  

\end{theorem}

\begin{proof}
If $P$ is a finitely generated projective $L_{\mathbb{F}}(\Gamma)$-module then $P\cong \oplus_{i=1}^n v_iL_{\mathbb{F}}(\Gamma)$ by 
\cite[Theorem 3.5]{amp07}. Since $vL_{\mathbb{F}}(\Gamma) \neq 0$ for all $v \in V$, if $P$ is indecomposable then $n=1$. Either $P\cong wL_{\mathbb{F}}(\Gamma)$ with $w$ a sink or $P\cong vL_{\mathbb{F}}(\Gamma)$ with $ s^{-1}(v)=\{ e \}$, a singleton because $vL_{\mathbb{F}}(\Gamma) \cong \bigoplus_{se=v} teL_{\mathbb{F}}(\Gamma)$. Similarly, either $te$ is a sink or $s^{-1}(te)$ is a singleton. Continuing in this manner we see that either $P\cong wL_{\mathbb{F}}(\Gamma)$ for some sink $w$ or $P\cong vL_{\mathbb{F}}(\Gamma)$ with $v$ on a cycle with no exit (if we don't get to a sink we must have a repetition since $\Gamma$ is finite).\\

Conversely, if $w$ is a sink then $wL_{\mathbb{F}}(\Gamma)$ is simple by Proposition \ref{Path}(iii). If $C$ is a cycle with no exit and $v=sC$ then $M=vL_{\mathbb{F}}(\Gamma)$ is generated by $v \in Mv=vL_{\mathbb{F}}(\Gamma)v$. Theorem \ref{Ser2} implies that  $vL_{\mathbb{F}}(\Gamma)$ is an indecomposable $L_{\mathbb{F}}(\Gamma)$-module because $Mv= vL_{\mathbb{F}}(\Gamma)v$ is an indecomposable $vL_{\mathbb{F}}(\Gamma)v \cong \mathbb{F}[x,x^{-1}]$-module. Similarly, $vL_{\mathbb{F}}(\Gamma)$ is not simple since $\mathbb{F}[x,x^{-1}]$ is not a simple $\mathbb{F}[x,x^{-1}]$-module.\\

Distinct sinks (respectively, distinct cycles with no exits) give non-isomorphic (indecomposable) projective $L_{\mathbb{F}}(\Gamma)$-modules because the sink $w$ (respectively, the cycle $C$) is the unique one in the support $\Gamma_{\leadsto w}$ (respectively,  $\Gamma_{\leadsto C}$). Thus assigning the unique sink (respectively, cycle with no exit) to the fgip gives the one-to-one correspondences. 
 \end{proof}\\

\begin{corollary} \label{Sonuc}
A projective $L_{\mathbb{F}}(\Gamma)$-module $P$ is simple if and only if $P\cong wL_{\mathbb{F}}(\Gamma)$ for some sink $w$, which is the unique sink in $V_P$, the support of $P$. 
\end{corollary}
\begin{proof}
This follows from Theorem \ref{fgip} and Proposition \ref{Path}(iii).
\end{proof}\\

For each cycle $C$ let's fix a base vertex $v_C:=sC$. We define $ht(v_C)$ to be $ht(C)$. Let $H_n$ be the hereditary saturated closure of the set of sinks and $\{v_C \> \vert \> ht(C)<n\}$ and $J_n:= (H_n)$, the ideal generated by $H_n$. We define $\Gamma_n$ as the subgraph of $\Gamma$ obtained by deleting all vertices in $H_n$ and all arrows touching them (that is, $\Gamma_n:=\Gamma_{/H_n}$ using our previous notation in Proposition \ref{hereditary}(ii)). Thus we get an ascending filtration of the set of vertices
$$H_0:=\emptyset \subseteq  H_1 \subseteq H_2 \subseteq \cdots \subseteq H_d  \subset  H_{d+1}=V$$

and an ascending filtration of the algebra $L_{\mathbb{F}}(\Gamma)$ by ideals
$$ J_0:=0 \leq J_1 \leq J_2 \leq \cdots \leq J_d  < J_{d+1}= L_{\mathbb{F}}(\Gamma)$$

 and a descending filtration of $\Gamma$ by the subgraphs:
 
$$\Gamma \supseteq \Gamma_1 \supseteq \>\> \Gamma_2 \supseteq  \> \cdots \> \supseteq \Gamma_d \supset \Gamma_{d+1}:=\emptyset .$$

 Note that $L_{\mathbb{F}}(\Gamma_n)\cong L_{\mathbb{F}}(\Gamma)/ J_n $ by Proposition \ref{hereditary}(ii). Also 
 $H_n= J_n\cap V $ since if $v\in V\setminus H_n$ then its image in $L_{\mathbb{F}}(\Gamma)/ J_n \cong L_{\mathbb{F}}(\Gamma_n)$ is nonzero. The intersection of the ideal generated by $\{ w \in V \>\vert \> w \text{ a sink } \} \cup \{v_C \> \vert \> ht(C) < n\}$ with $V$ is a hereditary saturated subset of $V$ by Proposition \ref{hereditary}(i), containing its hereditary saturated closure, i.e., $H_n$. Hence $J_n$ is generated by $\{ w \in V \>\vert \> w \text{ a sink } \} \cup \{v_C \> \vert \> ht(C) < n\}$ for $n>0$. \\

If $C$ is a cycle in $\Gamma$ with $ht(C)=n$ then $v_C=sC$ is not in $H_n$ since $C^{\infty}$ is an infinite path that does not meet any sink or a successor of $\{ v_D \> \vert \> ht(D) < n  \}$ by the discussion preceding Proposition \ref{hereditary}.
If $ht(\Gamma)=d>0$ then there is a maximal cycle $C$ in $\Gamma$ with $ht(C)=d$, so $H_d \neq V =H_{d+1}$. Similarly, $H_n \neq H_{n+2}$ hereditary for $0\leq n < d$, however $H_n=H_{n+1}$ is a possibility (for instance, for every $n\in \mathbb{N}$ there is a quantum sphere $qS^d$ given in Example \ref{ornek} with $H_n=H_{n+1}$). Also, $\Gamma_d $ is non-empty because it has $v_C$ for some cycle $C$ with $ht(C)=d=ht(\Gamma)$ and $\Gamma_n \neq \Gamma_{n+2}$ for $0\leq n < d$ but $\Gamma_n = \Gamma_{n+1}$ is possible.\\

If $C$ is a cycle with $ht(C)=n$ then $v_C$ is a vertex of $\Gamma_n$ because $v_C \notin H_n$. Hence $v_C$ in $L_{\mathbb{F}}(\Gamma_n) \cong L_{\mathbb{F}}(\Gamma)/ J_n $ is nonzero, that is, $v_C \notin J_n$. As above,  $J_d \neq  L_{\mathbb{F}}(\Gamma)$ and $J_n\neq J_{n+2}$  for $0\leq n< d$  but $J_n=J_{n+1}$ is possible. $ GK dim L_{\mathbb{F}}(\Gamma)=d$ by Theorem \ref{height}.\\   

If $te$ is in $H_n$ for all $e \in E$ with $se=v$  then $v \in H_n$ because $H_n$ is saturated. Hence a minimal vertex of $\Gamma_n$ must be on a cycle, i.e., $\Gamma_n$ has no sinks for $n>0$. The cycles with no exit in $\Gamma_n$ for $n>0$, come from the cycles of $\Gamma$ of height $n$ or $n+1$ by the definition of height.

\begin{remark}
Our ascending chain of ideals is a refinement of the chain of ideals defined in \cite{aajz13} since $I_m=J_{2m}$.  
\end{remark}

Next we define a descending filtration by Serre subcategories of the category of  $L\Gamma := L_{\mathbb{F}}(\Gamma)$-modules: \\

Let $\mathfrak{M}_\Gamma^0 := Mod_{L\Gamma}$ and    
%\textcolor{blue}{fgip}: finitely generated indecomposable projective \\\\
$\mathfrak{M}_\Gamma^1$ be the full subcategory of  $Mod_{L\Gamma}$ with objects $M$ such that $Hom^{L\Gamma}(P,M)=0 $ for all simple projective modules $P$ in  $Mod_{L\Gamma}$. Hence $Hom^{L\Gamma}(wL\Gamma, M)=0$ for all sinks $w$ by Corollary \ref{Sonuc}. Also, $Mw \cong Hom^{L\Gamma}(wL\Gamma, M)$ as explained in the proof of Theorem \ref{Ser2} and so $Mw=0$ for all sinks $w$. By definition, the ideal $J_1$ annihilates all such $L\Gamma$-modules $M$ and all $M$ annihilated by $J_1$ is in $\mathfrak{M}_\Gamma^1$ since $J_1$ is generated by the sinks in $\Gamma$. Thus $\mathfrak{M}_\Gamma^1$ may be identified with the category of $L\Gamma_1\cong L\Gamma /J_1$-modules. We can recover $J_1$ from $\mathfrak{M}_\Gamma^1$ as the intersection of the annihilators of all $M$ in $\mathfrak{M}_\Gamma^1$ since $L\Gamma/J_1$ is in $\mathfrak{M}_\Gamma^1$ and  $Ann(L\Gamma/J_1)=J_1$. \\

Let $\mathfrak{M}_{\Gamma}^{2}$ be the full subcategory of  $\mathfrak{M}_{\Gamma}^{1}$ with objects $M$ such that $Hom^{L\Gamma}(P,M)=0 $ for all finitely generated indecomposable projectives $P$ in  $Mod_{L\Gamma} =\mathfrak{M}_{\Gamma}^{0}$. Similar to the discussion above the modules $M$ in $\mathfrak{M}_{\Gamma}^{2}$ are those with $Mv\cong Hom^{L\Gamma}(vL\Gamma, M)=0$ for all $v_C$ on all cycles $C$ with no exit by Proposition \ref{fgip}(ii) and
$Mw=0$ for all sinks $w$, since $\mathfrak{M}_{\Gamma}^{2}$ is a subcategory of $\mathfrak{M}_{\Gamma}^{1}$. Also 
$\mathfrak{M}_\Gamma^2$ may be identified with the category of $L\Gamma_2\cong L\Gamma /J_2$-modules. We can recover $J_2$ from $\mathfrak{M}_\Gamma^2$ as the intersection of the annihilators of all $M$ in $\mathfrak{M}_\Gamma^2$ since $L\Gamma/J_2$ is in $\mathfrak{M}_\Gamma^2$ and  $Ann(L\Gamma/J_2)=J_2$.\\

Let $\mathfrak{M}_{\Gamma}^{n+1}$ for $n>0$ be the full subcategory of  $\mathfrak{M}_{\Gamma}^{n}$ with objects $M$ such that $Hom^{L\Gamma}(P,M)=Hom^{L\Gamma_{n-1}}(P,M)=0 $ for all finitely generated indecomposable projectives $P$ in  $\mathfrak{M}_{\Gamma}^{n-1}$. As above, $\mathfrak{M}_\Gamma^{n+1}$ may be identified with the category of $L\Gamma_{n+1}\cong L\Gamma /J_{n+1}$-modules. We can recover $J_{n+1}$ from $\mathfrak{M}_\Gamma^{n+1}$ as the intersection of the annihilators of all $M$ in $\mathfrak{M}_\Gamma^{n+1}$ since $L\Gamma/J_{n+1}$ is in $\mathfrak{M}_\Gamma^{n+1}$ and  $Ann(L\Gamma/J_{n+1})=J_{n+1}$.\\

If $P$ is a projective module then $Hom(P, \underline{\>\>})$ is an exact functor, therefore $\mathfrak{M}_{\Gamma}^n$ is a Serre subcategory of $Mod_{L\Gamma}$ for all $n$. Since $\mathfrak{M}_{\Gamma}^n$ is identified with $Mod_{L\Gamma /J_n}$  we have a finite descending filtration of Serre subcategories: \\

$\mathfrak{M}_{\Gamma}^0 \supseteq \mathfrak{M}_{\Gamma}^1 \supseteq \cdots \supseteq \mathfrak{M}_{\Gamma}^d \supset \mathfrak{M}_{\Gamma}^{d+1}= \textbf{0} \qquad $
where $d=GKdim L(\Gamma)$\\

An object of a Serre subcategory $\mathfrak{C}$ is simple or indecomposable or finitely generated  in the ambient category if and only if it is so in $\mathfrak{C}$. However, projectives in $\mathfrak{C}$ may not be projective in the ambient category. In fact, the finitely generated projective modules  that were necessary to define $\mathfrak{M}_{\Gamma}^n $ for all $n>2$ are not projective as $L\Gamma$-modules, but they are projective in $\mathfrak{M}_{\Gamma}^{n-2} $, that is, as $L\Gamma/J_{n-2}$-modules. \\

Note that $v_C$ is in $H_n$ if and only if all the vertices on $C$ are  in $H_n$, because $H_n$ is hereditary. Therefore $H_n$, so $\Gamma_n$, $J_n$ and $\mathfrak{M}_{\Gamma}^n$ are all independent of the choices of the base vertices $v_C$ for all cycles $C$ in $\Gamma$.\\

\subsection{Morita Invariants of $L_{\mathbb{F}}(\Gamma)$}

Let $U_0:=U_0(\Gamma)$ be the isomorphism classes of simple projectives in $\mathfrak{M}_{\Gamma}^0 =Mod_{L\Gamma}$, let $U_1:=U_1(\Gamma)$ be the isomorphism classes of nonsimple finitely generated indecomposable projectives in $\mathfrak{M}_{\Gamma}^0$ and when $n>0$ let 
 $U_{n+1}:=U_{n+1}(\Gamma)$ be the isomorphism classes of finitely generated indecomposable projectives in $\mathfrak{M}^n_{\Gamma}$ that are  not in $U_n:=U_n(\Gamma)$.\\

We have an ascending chain of finite sets:
 $$U_0 \subseteq U_0 \cup U_1 \subseteq \cdots \subseteq U_0\cup U_1 \cup \cdots \cup U_d $$
Applying Theorem \ref{fgip} to $L\Gamma /J_n$ for each $n$ we get 
 a one-to-one correspondence between $U_0$ and the sinks in $\Gamma$, also $U_1$ and the cycles with no exits in $\Gamma$, in general, $U_n$ and the cycles of height $n$ in $\Gamma$ for $1\leq n\leq d$. \\

  \begin{remark}
  If $u$ and $v$  are two vertices on a cycle $C$ of height $n$ then $uL\Gamma/uI_n \cong uL\Gamma_n \cong vL\Gamma_n \cong vL\Gamma/vI_n $ and this isomorphism class is a unique element of $U_n$. The middle isomorphism holds because $C$ is a cycle with no exit in $\Gamma_n$: the homomorphism $teL\Gamma_n \stackrel{e \cdot}{\longrightarrow} seL\Gamma_n$ is an isomorphism with inverse $e^* \cdot$ when $e$ is an arrow on $C$, because $e^*e=te$ by (CK1) and $ee^*=se$ by (CK2) since $C$ has no exit in $\Gamma_n$. However, $uL\Gamma$ may not be isomorphic to $vL\Gamma$ as the following example shows. 
 \end{remark}

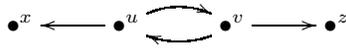
\begin{example}

$$  
 \xymatrix{ {\bullet}^{x}& \ar [l] {\bullet}^{u} \ar@/^0.6pc/[r] & {\bullet}^{v} \ar [r]    \ar@/^0.5pc/[l] &{\bullet}^{z}} $$ \\

\noindent
$[uL\Gamma]\neq [vL\Gamma] $ in $K_0(L\Gamma)$, hence $uL\Gamma \ncong vL\Gamma$.
\end{example}

 The subcategories $\mathfrak{M}_{\Gamma}^n$ of $Mod_{L_{\mathbb{F}}(\Gamma)}$ are defined completely category theoretically since simple, projective, indecomposable and finitely generated are categorically concepts. That is, if ${L_{\mathbb{F}}(\Gamma)}$ and ${L_{\mathbb{F}}(\Lambda)}$ are Morita equivalent then a functor giving an equivalence between $Mod_{L_{\mathbb{F}}(\Gamma)}$ and $Mod_{L_{\mathbb{F}}(\Lambda)}$ will send $\mathfrak{M}_{\Gamma}^n$ to $\mathfrak{M}_{\Lambda}^n$. Consequently, the ideals $J_n$ are also Morita invariants in the sense that they can be recovered from these subcategories. This generalizes the result in \cite{aajz13}
 stating that $J_{2m}=I_m$ are invariant under isomorphisms. \\

 The sets of isomorphism classes of finitely generated indecomposable modules $U_n$, corresponding to the sinks and cycles of $\Gamma$, equivalently the vertices of the  complete reduction of $\Gamma$, are also defined categorically. Therefore the vertices of the complete reduction of $\Gamma$ can be recovered from  $Mod_{L_{\mathbb{F}}(\Gamma)}$, hence they are also Morita invariants. \\

We also get a refinement of the  Gelfand-Kirillov dimension for Leavitt path algebras of polynomial growth as a Morita invariant.
\begin{theorem}\label{polinom}
If  $\Gamma$ is a finite digraph whose cycles are pairwise disjoint and $\mathbb{F}$ is a field then the polynomial $$G_{\Gamma}(z)=a_0+a_1z+a_2z^2+\cdots +a_dz^d$$ is a Morita invariant of $L_{\mathbb{F}}(\Gamma)$ and $deg G_{\Gamma}(z)= GKdimL_{\mathbb{F}}(\Gamma) $ where $a_0$ is the number of sinks in $\Gamma$ and $a_i$ is the number of 
cycles in $\Gamma$ of height $i$ for $i>0$. 
\end{theorem}
\begin{proof} This follows from the preceding discussion since $a_n= \vert U_n \vert$.
\end{proof}
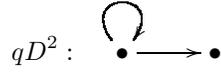
\begin{example}
$$\qquad \>  qD^2 : \quad \xymatrix{ {\bullet} \ar@(ul,ur)
 \ar@{->}[r]  & \bullet  } \qquad  \quad \quad \qquad G_{qD^2}(z)=1+z^2$$\\
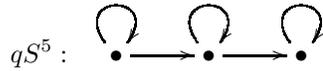
$$\qquad qS^5 : \quad \xymatrix{ {\bullet} \ar@(ul,ur)
 \ar@{->}[r]   &{\bullet} \ar@(ul,ur)
 \ar@{->}[r]    &  \bullet \ar@(ul,ur) } \qquad \quad \quad \quad G_{qS^5}(z)=z+z^3+z^5$$
\end{example}

 In order to recover the partial order $\leadsto$ on the vertices of the complete reduction of $\Gamma$, equivalently the sinks and the cycles of $\Gamma$, as well as some of the arrows of the complete reduction of $\Gamma$, we need to use the structure of simple $L(\Gamma)$-modules and their extensions in Section 4.\\

 The {\bf height of a simple $L\Gamma$-module} $M$, denoted by $ht(M)$, is  the largest $n$ such that it is in $\mathfrak{M}_{\Gamma}^n$. A simple module $M$ has $ht(M)=0$ if and only if $M\cong wL\Gamma$ for some sink $w$. If $C$ is a cycle in $\Gamma$ with $n=ht(C)$ then $M_C^{\lambda}:=v_CL\Gamma/(v_C-\lambda C^*)L\Gamma \cong v_CL\Gamma_n/(v_C-\lambda C^*)L\Gamma_n $ for  $0\neq \lambda \in \mathbb{F}$ is a simple module with $V_{M_C^{\lambda}}= V_{\leadsto C}$ by Corollary \ref{Classification 2}, hence  $ht(M_C^{\lambda})=n$. There is a unique isomorphism class $[P_C]$ in $U_n$, corresponding to the cycle $C$,
 with $Hom^{L\Gamma}(P_C, M_C^{\lambda})\neq 0$. (We can associate a unique sink or a cycle to each simple $L\Gamma$-module, but there are infinitely many simples associated to each cycle corresponding to the irreducible polynomials in $\mathbb{F}[x]$ with constant term 1.) If $M_D$ is another simple module defined by the cycle $D$ 
 then $C\leadsto D$ if and only if $Ext (M_C^{\lambda}, M_D^{\lambda})\neq 0$; If $w$ is a sink then $C\leadsto w$ if and only if $Ext(M_C^{\lambda}, wL\Gamma)\neq 0$ by Theorem \ref{katlilik}. \\

If $M=v_CL\Gamma / f(C^*)L\Gamma$ is simple then $(deg f(x))^2= (dim^{\mathbb{F}}(Mv_C))^2=dim^{\mathbb{F}} Ext (M, M)$ by Theorem \ref{katlilik}, hence $deg f(x)$ is determined categorically and so it is a Morita invariant. When $deg f(x)=1$, so $f(x)=1-\lambda x$, the defining polynomial of the simple module corresponding to $M_C^{\lambda}$ via an equivalence of categories must be $1-\mu x$ but $\mu $ may not equal $\lambda $. In fact, using a Morita equivalence given by a gauge isomorphism sending $C$ to a scalar of multiple of $C$ will interchange any nonzero $\lambda$ and $\mu$. \\

Let's consider finite digraphs $\Gamma$ and $\Lambda$ whose cycles are pairwise disjoint, with $L_{\mathbb{F}}(\Gamma)$ Morita equivalent to $L_{\mathbb{F}}(\Lambda)$. We may assume that $\Gamma$ and $\Lambda$ are completely reduced by Theorem \ref{Reduced}, so the vertices of $\Gamma$ are either sinks or $v_C$ which is the "base" vertex of the loop $C$. If $\mathcal{F}$ is a functor from $Mod_{L\Gamma}$ to  $Mod_{L\Lambda}$ giving a Morita equivalence then $\mathcal{F}$ maps $\mathfrak{M}_{\Gamma}^n$ to $\mathfrak{M}_{\Lambda}^n$ and $U_n(\Gamma)$ to $U_n(\Lambda)$ as explained above. The sinks in $\Gamma$ correspond to the isomorphism classes of simple projective $L_{\mathbb{F}}(\Gamma)$-modules, that is, elements of $U_0(\Gamma)$. Similarly, $\mathcal{F} (U_0(\Gamma))=U_0(\Lambda)$ corresponds to the sinks in $\Lambda$, yielding a one-to-one correspondence between the sinks of $\Gamma$ and the sinks of $\Lambda$. Each element of $U_n(\Gamma)$ corresponds to a unique cycle in $\Gamma$, similarly yielding a one-to-one correspondence between the cycles of height $n$ in $\Gamma$ and the cycles of height $n$ in $\Lambda$ via $\mathcal{F}$. For each cycle $C$ in $\Gamma$ identified with $[P_C] \in U_n(\Gamma)$, we pick a simple module $M$ associated to $C$ with $dim^{\mathbb{F}} (Mv_C)=1$, for instance $M_C=v_CL\Gamma /(v_C- C^*)L\Gamma$. Since $Hom^{L\Gamma}(P_C, M) \cong Hom^{L\Lambda}(\mathcal{F}(P_C), \mathcal{F}(M))$ we see that $\mathcal{F}(M)$ is associated with $[\mathcal{F}(P_C)]$. Also $Ext(M,M) \cong Ext (\mathcal{F}(M), \mathcal{F}(M))$ and so $dim^{\mathbb{F}} Ext(\mathcal{F}(M), \mathcal{F}(M))=dim^{\mathbb{F}}Ext(M,M)=1$. Thus $\mathcal{F}(M)$ is defined by a polynomial of degree 1 by Theorem \ref{katlilik}. Using Theorem \ref{katlilik} again, we see that the poset of the sinks and the cycles in $\Gamma$ (under $\leadsto$) is isomorphic to the poset of the sinks and the cycles in $\Lambda$.\\

We have seen that the Hasse diagram of the partial order $\leadsto$ on the sinks and the cycles of $\Gamma$ is a Morita invariant. We can do better: If the cycle $C$ in $\Gamma$ covers the cycle $D$ or the sink $w$ then there is no other cycle in between, so $v_C$ can only connect to $v_D$ or $w$ via paths of length 1, that is, arrows, in the complete reduction of $\Gamma$. The {\bf weighted Hasse diagram} of $\Gamma$ is the Hasse diagram of the poset of sinks and cycles in $\Gamma$ with each sink marked (to distinguish sinks from cycles with no exits) and each edge corresponding to  $v_C$ covering $v_D$, respectively the sink $w$, labelled with the number of arrows from $v_C$ to $v_D$, respectively to $w$. This number is $dim^{\mathbb{F}}Ext (M_C, M_D)= dim^{\mathbb{F}}Ext (\mathcal{F}(M_C), \mathcal{F}(M_D))$ or $dim^{\mathbb{F}}Ext (M_C, wL\Gamma)= dim^{\mathbb{F}}Ext (\mathcal{F}(M_C), \mathcal{F}(wL\Gamma))$. This discussion proves:

\begin{theorem} \label{Hasse}
If $\mathbb{F}$ is a field and $\Gamma$ is a finite digraph whose cycles are pairwise disjoint then the weighted Hasse diagram of $\Gamma$ is a Morita invariant of $L_{\mathbb{F}}(\Gamma)$. 
\end{theorem}

Note that the weighted Hasse diagram depends only on the digraph $\Gamma$ and is independent of the coefficient field $\mathbb{F}$.

\subsection{Classification of $L_{\mathbb{F}}(\Gamma)$ with $GKdim L_{\mathbb{F}}(\Gamma)< 4$ }

 An arrow $e$ in a completely reduced finite digraph whose cycles are pairwise disjoint is called a {\bf shortcut} if there is a vertex  strictly between $se$ and $te$, equivalently if $se$ does not cover $te$ with respect to $\leadsto$. In the previous subsection we saw how to recover the vertices, the sinks, the cycles (=loops) and all arrows that are not shortcuts 
 of the complete reduction of a finite digraph $\Gamma$ whose cycles are pairwise disjoint from $Mod_{L\Gamma}$. This is exactly the information contained in the weighted Hasse diagram of $\Gamma$. If $GKdim L\Gamma <4$ then there is no room for shortcuts in the complete reduction of $\Gamma$ by Theorem \ref{height}.

 \begin{theorem} \label{<4}
If $\Gamma$ and $\Lambda$ are finite digraphs whose cycles are pairwise disjoint and $GKdim L\Gamma <4$ then $L\Gamma$ and $L\Lambda$ are Morita equivalent if and only if their complete reductions are isomorphic digraphs.
 \end{theorem}
 
\begin{proof}
If $L\Gamma$ and $L\Lambda$ are Morita equivalent then $GKdim L\Lambda <4$ by Theorem \ref{polinom}. They have the same weighted Hasse diagram since they are Morita equivalent. Their complete reductions are isomorphic because they are determined by their weighted Hasse diagrams when the Gelfand-Kirillov dimension is less than 4. Conversely, if their complete reductions are isomorphic then $L\Gamma$ and $L\Lambda$ are Morita equivalent by Theorem \ref{Reduced}.
\end{proof}
 \\
 
 Theorem \ref{<4} implies that the weighted Hasse diagram is a complete Morita invariant when the Gelfand-Kirillov dimension is less than 4.  The following examples show that this does not hold in general.

 \begin{example} \label{quantumdisk} Consider the following digraph $\Gamma $ and the quantum 4-disk $qD^4$ obtained from $\Gamma$ by deleting the shortcut $e$.\\

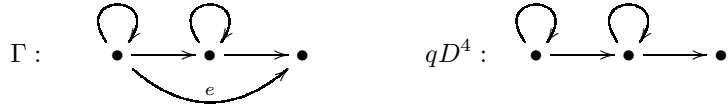
$$\Gamma : \qquad  \xymatrix{ {\bullet} \ar@(ul,ur) \ar@/_1.5pc/[rr]^e \ar@{->}[r]   & \bullet \ar@(ul,ur) \ar@{->}[r] &\bullet }  \qquad \qquad  qD^4 :   \quad  \xymatrix{ {\bullet} \ar@(ul,ur)   \ar@{->}[r]   & \bullet \ar@(ul,ur) \ar@{->}[r] &\bullet } $$

         %        $$\qquad \qquad \Gamma \qquad \qquad    \qquad \qquad  \qquad \qquad  qD^4 $$

          \noindent
          Both of these digraphs are completely reduced and they are not isomorphic. However $L(\Gamma)\cong L(qD^4)$, so  $L(\Gamma)$ and $L(qD^4)$ are Morita equivalent.
 \end{example}

 \begin{example} The completely reduced digraphs $\Gamma$ and $\Lambda $ below have the same weighted Hasse diagram. Deleting the shortcut $e$ from $\Gamma$ we obtain $\Lambda$.  
 
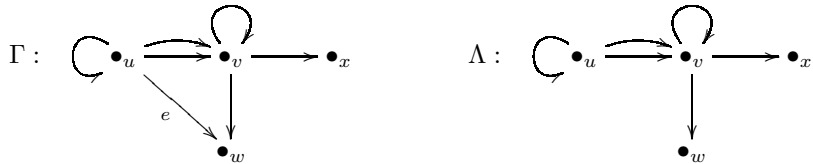
$$ \Gamma : \qquad \xymatrix{ {\bullet}_u \ar@(ul,dl)
 \ar@{->}[r] \ar@{->}[dr]_e  \ar@/^/[r]   & \ar@{->}[d] \bullet_v \ar@(ul,ur) \ar@{->}[r] &\bullet _x\\
                 & \bullet_w &} \qquad \qquad \Lambda :  \quad \quad    \xymatrix{ {\bullet}_u \ar@(ul,dl)
 \ar@{->}[r]  \ar@/^/[r]   &\ar@{->}[d] \bullet_v \ar@(ul,ur) \ar@{->}[r] &\bullet_x \\
                 & \bullet_w &} $$    
% $$\Gamma \qquad   \qquad \qquad \qquad \qquad \qquad \qquad \qquad \Lambda   \qquad \qquad $$
 
\noindent
$L(\Gamma)$ and $L(\Lambda)$ are not Morita equivalent since $K_0(L(\Gamma)) \not\cong K_0(L(\Lambda))$.\\

\end{example}
 
 To finish the classification of Leavitt path algebras of polynomial growth with Gelfand-Kirillov dimension $\geq 4$, up to Morita equivalence, what remains is to understand the contribution of 
 shortcuts to Morita type.\\\\

 \noindent
 \textbf{Acknowledgements}\\

Ayten Ko\c{c} was supported by TUBITAK (The Scientific and Technological Research Council of Turkey)  BIDEB 2219-International Post-Doctoral Research Fellowship during her visit to the University of Oklahoma where some of this research was done. She would like to thank her colleagues at the host institution for their hospitality.

$^*$ Department of Mathematics, \\ Gebze Technical  University,  Gebze, TÜRK\.{I}YE\\ 
E-mail: aytenkoc@gtu.edu.tr\\

$^{**}$ Department of Mathematics, \\ University of Oklahoma, Norman, OK, USA \\
E-mail: mozaydin@ou.edu

  \end{document}